\documentclass[11pt,a4paper]{amsart}

\usepackage{amsmath,amssymb}
\usepackage{amsthm} 
\usepackage{amsfonts}
\usepackage[utf8]{inputenc} 
\usepackage{a4wide}
\usepackage{graphicx}
\usepackage{enumitem}
\usepackage{float}
\usepackage{tikz}
\usetikzlibrary{positioning, decorations.pathmorphing}
\usepackage{adjustbox}
\usepackage{color}
\usepackage{mathrsfs}
\usepackage{textcomp}

\usepackage[style=numeric, 
  sorting=nyt,
  giveninits=true,
  defernumbers=true,
  backend=biber]{biblatex}

\DeclareFieldFormat{titlecase}{#1}

\usetikzlibrary{positioning, decorations.pathmorphing}

\theoremstyle{plain}
\newtheorem{theorem}{Theorem}[section]
\newtheorem{proposition}[theorem]{Proposition}
\newtheorem{lemma}[theorem]{Lemma}
\newtheorem{corollary}[theorem]{Corollary}

\theoremstyle{definition}
\newtheorem{definition}[theorem]{Definition}
\newtheorem{example}[theorem]{Example}

\theoremstyle{remark}
\newtheorem{remark}[theorem]{Remark}

\numberwithin{figure}{section}
\numberwithin{theorem}{section}
\numberwithin{equation}{section}

\begin{document}

\date{\today}

\title[Boolean lattices formed by partitions]{On Boolean sublattices of a partition lattice}

\author{Stephan Foldes}
\address[S.~Foldes]{Institute of Informatics, University of Miskolc, 3515 Miskolc-Egyetemv\'{a}ros, Hungary}
\email{foldes.stephan@uni-miskolc.hu}

\author{S\'{a}ndor Radeleczki}
\address[S.~Radeleczki]{Institute of Mathematics, University of Miskolc, 3515 Miskolc-Egyetemv\'{a}ros, Hungary}
\email{sandor.radeleczki@uni-miskolc.hu}



\begin{abstract}
We investigate maximal Boolean sublattices of the
partition lattice Part$(U)$ of a finite universe $U$. First, any
largest size Boolean sublattice of Part$(U)$ can be formed using all
partitions whose blocks are subtrees of a tree with vertex set $U$.
It is shown that a maximal Boolean sublattice of Part$(U)$ always
contains the least and the largest elements of Part$(U)$,
$\triangle$ and $\triangledown$. Boolean sublattices $\mathcal{B}$
of Part$(U)$ containing $\triangle$ are characterized by a certain
condition imposed on the cycles of a linear hypergraph corresponding
to the atoms of $\mathcal{B}$ on the set $U$. We show that
$\mathcal{B}$ can be extended to a Boolean sublattice of Part$(U)$
with a largest size, if and only if the hypergraph induced by its
atoms is a hypertree. This is the case when $\mathcal{B}$ is formed
by all the partitions whose blocks are intervals in a generalized
sense.

The main result states that all partition lattices of height at least three have maximal Boolean sublattices for all possible dimensions $d\geq3$.
\end{abstract}

\maketitle

\section{Motivation}\label{sec1}

\medskip

In this paper we study some special Boolean sublattices of a finite partition
lattice. The motivation for this article comes from the study of the so-called
interval decompositions. Decompositions into intervals were first studied by
Hausdorff \cite{Hausdorff1908}, in the context of linearly ordered sets, then the interval
concept was extended to partially ordered sets and graphs (Sabidussi \cite{Sabidussi1961}),
and appeared in particular in the study of comparability graphs (Gallai \cite{Gallai1967}). A
general, abstract theory of interval decompositions was first presented by
M\"{o}hring and Radermacher \cite{Mohring1985,MohringRadermacher1984}. 
\medskip

A closure system $(U,\mathcal{Q})$, $\mathcal{Q}\subseteq\wp(U)$ is called
\emph{algebraic} if the union of any chain of closed sets is closed. An
\emph{interval system }$(U,\mathcal{I})$ was defined in \cite{FoldesRadeleczki2004} as an algebraic
closure system with the following properties:
\smallskip

(I$_{\text{0}}$) $\{x\}\in\mathcal{I}$, for all $x\in U$ and $\emptyset\in$
$\mathcal{I}$,

(I$_{\text{1}}$) $A,B\in \mathcal{I}$ and $A\cap B\neq$ $\emptyset$ imply
$A\cup B\in$ $\mathcal{I}$,

(I$_{\text{2}}$) for any $A,B\in \mathcal{I}$ the relations $A\nsubseteq B$
and $B\nsubseteq A$ imply $A\setminus B\in\mathcal{I}$

(and $B\setminus A\in\mathcal{I}$).

\smallskip
\noindent A \emph{decomposition} in a closure system $(U,\mathcal{Q})$ is a
partition $\pi=\{A_{i}$ $\mid i\in I\}$ of the set $U$ such that $A_{i}%
\in\mathcal{Q}$, for all $i\in I$. If $(U,\mathcal{Q})$ is an interval system,
then $\pi$ is called an \emph{interval decomposition}. The set of all
decompositions in $(U,\mathcal{Q})$ is denoted by $\mathcal{D}(U,\mathcal{Q})$.

Let Part$(U)$ denote the lattice of all partitions of $U$ and $\wp(U)$ its
power set. The least element of Part$(U)$ is the partition $\triangle:=$
$\{\{x\}\mid x\in U\}$, and its greatest element is the partition $\nabla$
having as a single block the whole set $U$. Since $\mathcal{D}(U,\mathcal{Q})$
is included in$\ $Part$(U)$, it is ordered by refinement, i.e. for any
$\pi_{1},\pi_{2}\in\mathcal{D}(U,\mathcal{Q})$, $\pi_{1}\leq\pi_{2}$ holds if
and only if every block of $\pi_{2}$ is the union of some blocks $\pi_{1}$.
Then $\mathcal{D}(U,\mathcal{Q})$ is a complete lattice with the greatest
element $\nabla$, where the $\wedge$ operation is the same as in Part$(U)$
(see e.g. \cite{FoldesRadeleczki2004}).

\begin{remark}\label{rem. 1.1} Let $(U,\mathcal{Q})$ be a closure system that satisfies
condition (I$_{\text{0}}$). Then $\triangle$ is the least element of
$\mathcal{D}(U,\mathcal{Q)}$, and to any $A\in\mathcal{Q}\setminus
\{\emptyset\}$ corresponds the decomposition
$\pi_{A}=\{A\}\cup\{\{x\}\mid
x\in U\setminus A\}$. If $\pi=\{A_{i}$ $\mid i\in I\}\in\mathcal{D}%
(U,\mathcal{Q})$, then%
\begin{equation}
\pi=\bigvee\{\pi_{A_{i}}\mid i\in I\}\text{,} \tag{1}%
\end{equation}

\noindent Note that in (1) the join $\bigvee\{\pi_{A_{i}}\mid i\in
I\}$ is the same in the lattices $\mathcal{D}(U,\mathcal{Q})$ and
Part$(U)$. Partitions of the form $\pi_{A}$, $A \subseteq U \setminus
\{\emptyset\}$ are called \emph{principal
partitions}. In this case the atoms in the lattice
$\mathcal{D}(U,\mathcal{Q})$ are principal partitions. The atoms
in Part$(U)$ are exactly the partitions $\pi_{\{a,b\}}$, with
$a,b\in U$, $a\neq b$.
\end{remark}

\noindent In \cite{FoldesRadeleczki2004} it was proved the following

\medskip

\begin{proposition}\label{prop. 1.2} Let $(U,\mathcal{Q})$ be an algebraic
closure system satisfying condition (I$_{\text{0}}$). Then $\mathcal{D}%
(U,\mathcal{Q})$ is a complete semimodular sublattice of Part$(U)$
if and only if $(U,\mathcal{Q})$ satisfies condition
(I$_{\text{1}}$).
\end{proposition}

\noindent Therefore, for an interval system $(U,\mathcal{I})$, the lattice
$\mathcal{D}(U,\mathcal{I})$ is always a complete semimodular sublattice of
Part$(U)$. The following example and proposition were the starting point of
our investigations.

\medskip
\begin{example}\label{ex. 1.3} If $T=(V,E)$ is a finite tree, then the
vertex sets of its subtrees form a closure system $(V,\mathcal{Q})$
which satisfies conditions (I$_{\text{0}}$) and (I$_{\text{1}}$).
Thus, in view of Proposition 1.2, $\mathcal{D}(V,\mathcal{Q})$ is a
finite sublattice of Part$(V)$. Moreover, in \cite{FoldesRadeleczki2016} we proved that
$\mathcal{D}(V,\mathcal{Q})$ is a Boolean lattice isomorphic to
$(\wp(E),\subseteq)$.
\end{example}

\noindent Let $L$ be a lattice with a least element $0$. An element $a\in L$
is an \emph{atom} if it covers $0$, that is, $0\prec a$. $L$ is
\emph{atomistic} if every element of $L$ is a join of atoms. It is well-known
that all the partition lattices and the finite Boolean lattices are atomistic.
The \emph{dimension} $d$ of a Boolean lattice $B$ is the number of its atoms
which is the same as its height; obviously, $\mid B\mid=2^{d}$, where $\mid
B\mid$ denotes the cardinality (size) of the set $B$. Observe that in the
previous example $\mathcal{D}(V,\mathcal{Q})$ is a Boolean lattice of size
$2^{\mid V\mid-1}$.

\medskip
\begin{proposition}\label{prop. 1.4} Let $V$ be an $n$-element set, $n\geq3$.
Then the following holds true:

\noindent(i) The largest size Boolean sublattices of Part$(V)$ are $n-1$ dimensional.

\noindent(ii) The Boolean sublattices of dimension $n-1$ correspond 
bijectively to the trees on the node set $V$, where to a tree $T$ the
members of the corresponding lattice are the partitions of $V$ into
connected node sets (subtrees) of $T$.
\end{proposition}

\begin{proof} (i) Because Part$(V)$ has height $n-1$, the 
Boolean sublattices of the largest size of Part $(V)$ have height, i.e., a
dimension at most $n-1$. On the other hand, for any tree $T$ on the node
set $V$ the corresponding lattice of partitions into connected
node sets forms a Boolean lattice of dimension $n-1$.

\noindent(ii) Take any $n-1$ dimensional Boolean sublattice $B$ of Part$(V)$.
As the height of $B$ equals to the height of Part$(V)$, the atoms of $B$ are
necessarily atoms of Part$(V)$, i.e. they have the form $\pi_{\{a,b\}}$, where
$a,b\in U$, $a\neq b$.

Since to each atom of $B$ corresponds a proper pair $A=\{a,b\}$ of
elements of $V$, the pairs corresponding to the atoms of $B$ define
a graph $G=(V,E)$ on the vertex set $V$. Clearly, $G$ has neither
loops nor multiple edges. We prove that $G$ cannot have a cycle.
Indeed, suppose that $A_{0}=\{a,b\}$ is an edge in such a cycle
$\mathcal{C}$. Since the vertices $a$ and $b$ also belong to the
neighboring edges of $A_{0}$ in $\mathcal{C}$, $A_{0}$ is included
in the union of some other edges $A_{1},...,A_{k}$ of $\mathcal{C}$.
This implies $\pi_{A_{0}}\leq\pi_{A_{1}}\vee...\vee\pi_{A_{k}}$,
where all these partitions are atoms of $B$. As $B$ is Boolean, its
atoms form an independent set, hence we get $\pi_{A_{0}}\wedge\left(
\pi_{A_{1}}\vee...\vee\pi_{A_{k}}\right) =\triangle$ - a
contradiction. Since $B$ is $n-1$ dimensional, $G$ has $n-1$ edges
corresponding to the $n-1$ atoms of $B$. Therefore, $G$ is a tree
$T$. Clearly, all partitions of $V$ into vertex sets of subtrees of
$T$, must be in $B$, because they are the joins of some atoms
$\pi_{A_{i}}$ of $B$, where $A_{i}$ is an edge in $E$. As these
partitions form an $n-1$ dimensional Boolean lattice, that Boolean
lattice must coincide with $B$.
\end{proof}

\noindent Closure systems $(V,\mathcal{Q})$ satisfying conditions
(I$_{\text{0}}$) and (I$_{\text{1}}$) are called \emph{semi-interval}
\emph{systems}. Hence, as a consequence of Proposition~\ref{prop. 1.4}. we obtain:
\medskip

\begin{corollary}\label{cor. 1.5} The largest size Boolean sublattices of
Part$(V)$ are lattices $\mathcal{D}(V,\mathcal{I})$ of semi-interval
decompositions.
\end{corollary}

\noindent In what follows, extending this result, we show that $\mathcal{D}%
(V,\mathcal{I})$ is a Boolean sublattice of Part$(U)$ if and only if
the atoms of the lattice $\mathcal{D}(V,\mathcal{I)}$ define a
hypertree on the set $U$.

\medskip

\section{Boolean sublattices and independent sets in Part(U)}\label{sec2}

In this section, we study arbitrary Boolean sublattices of Part$(U)$ and
the proper extensions of the Boolean sublattices of an arbitrary finite
lattice $L$. First, we will prove that for the case of a finite set $U$, any
Boolean sublattice of Part$(U)$ is included in a Boolean sublattice of it
containing $\triangle$ and $\triangledown$.

\noindent Let $H$ be a nonempty subset of a lattice $L$. A finite meet of
elements in $H$ is called a $\wedge$\emph{-element}. The set of all $\wedge
$-elements of $H$ is denoted by $H_{\wedge}$. The set $H_{\vee}$ of $\vee
$\emph{-elements} of $H$ is defined dually.

\medskip

Let $L$ be a lattice with a least element $0$. A nonempty subset $Z\subseteq
L\setminus\{0\}$ is called \emph{independent} if the relation
\begin{equation}
\left(  \bigvee X\right)  \wedge\left(  \bigvee Y\right)  =\bigvee(X\cap Y)
\tag{2}%
\end{equation}

\noindent holds for all finite subsets $X,Y$ of $Z$. In view of Gr\"{a}tzer
\cite{Gratzer2011}, if $Z$ is an independent set of the lattice $L$, then the mapping
$\varphi\colon X\rightarrow\bigvee X$, $X\subseteq Z$ is a lattice isomorphism
between the lattice $\wp_{f}(Z)$ of finite subsets of $Z$, and $\langle
Z\rangle$, the sublattice of $L$ generated by $Z$. Obviously, in this case
$\langle Z\rangle=Z_{\vee}$ and $\langle Z\rangle$ is a Boolean lattice,
whenever $Z$ is finite.

We say that an ordered pair $(a,b)\in L^{2}$ is a \emph{modular pair} and we
write $(a,b)\in M$ if, for all $c\in L$,%
\[
c\leq b\text{ implies }c\vee(a\wedge b)=(c\vee a)\wedge b\text{.}%
\]

\noindent We say that $(a,b)$ is a \emph{dual modular pair}, and we write
$(a,b)\in M^{\ast}$ if, for all $c\in L$,%
\[
c\geq b\text{ implies }c\wedge(a\vee b)=(c\wedge a)\vee b\text{.}%
\]

\noindent An element $a$ of a lattice $L$ is called \emph{left modular}, if
$(a,x)\in M$ holds for all $x\in L$; $a\in L$ is \emph{right modular}, if
$(x,a)\in M$ holds for all $x\in L$. We say that $a$ is a \emph{modular
element} of $L$ if it is a right modular and left modular element of $L$ in
the same time. For instance, any atom of a lattice $L$ is a right modular
element in $L$, and any principal partition (and hence any atom) is a modular
element of the lattice Part$(U)$. If $d$ is a dual atom in $L$, then $(x,d)\in
M^{\ast}$, for all $x\in L$ (see e.g. Stanley \cite{Stanley1972} or Stern \cite{Stern1999}).

\medskip

\begin{lemma}\label{lemma 2.1} Let $L$ be a finite lattice with a least element
$0$, $Z\subseteq L$ an independent subset, and $m:=\bigvee Z$. If an
element $t\in L\setminus\{0\}$ satisfies the relations $t\wedge m=0$
and $(t,m)\in M$, then $Z\cup\{t\}$ is also independent.
\end{lemma}

\begin{proof} Let us consider the mapping $\varphi_{t}(x)=x\vee t$,
$x\in Z_{\vee}$. Then $x\in Z_{\vee}$ implies $x\leq m$. Observe that
$\varphi_{t}$ is an order-isomorphism between the posets $(Z_{\vee},\leq)$ and
$(\varphi_{t}\left(  Z_{\vee}\right)  ,\leq)$. Indeed, $\varphi_{t}$ is
order-preserving. Conversely, $\varphi_{t}(x_{1})\leq\varphi_{t}(x_{2})$
implies $\left(  x_{1}\vee t\right)  \wedge m\leq\left(  x_{2}\vee t\right)
\wedge m$. Since $(t,m)\in M$ and $x_{1},x_{2}\leq m$, we have $\left(
x_{i}\vee t\right)  \wedge m=x_{i}\vee(t\wedge m)=x_{i\text{ }}$,
$i\in\{1,2\}$, therefore, we get $x_{1}\leq x_{2}$.

As $\varphi_{t}\colon Z_{\vee}\rightarrow\varphi_{t}\left(  Z_{\vee}\right)  $
is onto, it is an order-isomorphism. Since $Z_{\vee}=\langle Z\rangle$,
$(Z_{\vee},\leq)$ and $(\varphi_{t}\left(  Z_{\vee}\right)  ,\leq)$ are
lattices and $\varphi_{t}\colon Z_{\vee}\rightarrow\varphi_{t}\left(  Z_{\vee
}\right)  $ is a lattice isomorphism. Now, we are going to prove that relation
(1) holds for any finite subset $X,Y\subseteq Z\cup\{t\}$. If $X,Y\subseteq
Z$, this is clear, because $Z$ is independent. If $t\in X\cap Y$, then $X,Y$
can be written in the form $X=X_{1}\cup\{t\}$ and $Y=Y_{1}\cup\{t\}$, where
$X_{1},Y_{1}\subseteq Z$. Now we get:

\smallskip

$\left(  \bigvee X\right)  \wedge\left(  \bigvee Y\right)  =\left[  \left(
\bigvee X_{1}\right)  \vee t\right]  \wedge\left[  \left(  \bigvee
Y_{1}\right)  \vee t\right]  =\varphi_{t}\left(  \bigvee X_{1}\right)
\wedge\varphi_{t}\left(  \bigvee Y_{1}\right)  =$

$\varphi_{t}\left(  \left(  \bigvee X_{1}\right)  \wedge\left(  \bigvee
Y_{1}\right)  \right)  =\varphi_{t}\left(  \bigvee(X_{1}\cap Y_{1})\right)
=\left(  \bigvee(X_{1}\cap Y_{1})\right)  \vee t=$

$\bigvee\left(  (X_{1}\cap Y_{1})\cup\{t\}\right)  =\bigvee\left(  \left(
X_{1}\cup\{t\}\right)  \cap\left(  Y_{1}\cup\{t\}\right)  \right)
=\bigvee(X\cap Y)$.

\smallskip

\noindent Finally, w.l.o.g. we can assume that $t\in X$ and $t\notin Y$. Then
$X=X_{1}\cup\{t\}$ and $X_{1},Y\subseteq Z$. Since $\left(  \bigvee Y\right)
\leq m$ we get:

$\left(  \bigvee X\right)  \wedge\left(  \bigvee Y\right)  =\left[  \left(
\bigvee X_{1}\right)  \vee t\right]  \wedge\left(  \bigvee Y\right)  \wedge
m=\left(  \left[  \left(  \bigvee X_{1}\right)  \vee t\right]  \wedge
m\right)  \wedge\left(  \bigvee Y\right)  $.

\noindent Because $(t,m)\in M$ and $\bigvee X_{1}\leq m$, we have

$\left(  \left[  \left(  \bigvee X_{1}\right)  \vee t\right]  \wedge m\right)
=\left(  \bigvee X_{1}\right)  \wedge(t\wedge m)=\bigvee X_{1}$.

\noindent As $X\cap Y=X_{1}\cap Y$, this result implies:

$\left(  \bigvee X\right)  \wedge\left(  \bigvee Y\right)  =\left(  \bigvee
X_{1}\right)  \wedge\left(  \bigvee Y\right)  =\bigvee(X_{1}\cap
Y)=\bigvee(X\cap Y)$,

\noindent proving (1), again. 
\end{proof}

\medskip

\begin{proposition}\label{prop. 2.2} Let $U\neq\emptyset$ be a finite set and
$\mathcal{B}$ a Boolean sublattice of Part$(U)$.

\noindent(i) If $\triangledown\notin\mathcal{B}$, then $\mathcal{B}$ is
contained in a strictly larger Boolean sublattice $\mathcal{B}^{^{\prime}}$
of Part$(U)$.

\noindent(ii) If $\triangledown\in\mathcal{B}$ but $\triangle\notin
\mathcal{B}$, then $\mathcal{B}$ is contained in a strictly larger Boolean
sublattice $\mathcal{B}^{^{\prime}}$ of Part$(U)$.
\end{proposition}

\begin{proof} Denote by $m$ the top element of $\mathcal{B}$, and
by $b$ its least element.
\noindent(i) We may assume $m\neq\triangle$, for otherwise the statement holds
obviously. Let $a$ be an atom of Part$(U)$ with $a\nleq m$ and $t=b\vee a$.
Since Part$(U)$ is a semimodular lattice we get $b\prec t$. Then clearly,
$t\wedge m=b$. As the principal filter $F=\{\pi\in\ $Part$(U)\mid\pi\geq b\}$
of Part$(U)$ is isomorphic to a partition lattice, and $t$ is an atom in $F$,
$(t,\pi)\in M$, holds for all $\pi\in F$. In particular, $(t,m)\in M$. Since
$\mathcal{B}$ is a finite Boolean lattice, the set $A(\mathcal{B})$ of its
atoms is independent. Now, by applying Lemma~\ref{lemma 2.1}, we get that $A(\mathcal{B}%
)\cup\{t\}$ is a (finite) independent system. Hence, the sublattice
$\mathcal{B}^{^{\prime}}\supseteq\mathcal{B}$ of Part$(U)$ generated by
$A(\mathcal{B})\cup\{t\}$ is a Boolean lattice. Since $t\in\mathcal{B}%
^{\prime}\setminus\mathcal{B}$, our statement holds.

\noindent(ii) We may assume\textit{ }$b\neq\triangledown$. Let $X$ be a
non-singleton block of $b$, $x\in X$, and denote by $t$ the partition with the
blocks $\{x\}$ and $U\setminus\{x\}$. \ Then $t$ is a dual atom of Part$(U)$
and $t\vee b=\triangledown$. We show that $(t,b)\in M^{\ast}$. Since $t$ is a
principal partition, we have $(t,p)\in M$, for all $p\in\ Part(U)$. Then
$(t,p)\in M^{\ast}$, for all $p\in\ $Part$(U)$ cf. Maeda \& Maeda \cite[Lemma 1.2]{MaedaMaeda2012}. In particular, $(t,b)\in M^{\ast}$, and this yields $(t,b)\in M$ in the
dual Part$(U)^{(d)}$ of the lattice Part$(U)$.

The set $D(\mathcal{B})$ of dual atoms of $\mathcal{B}$ is an independent
system in Part$(U)^{(d)}$, where their join $\bigvee^{(d)}D(\mathcal{B})$
equals $b$, and $t\wedge^{(d)}b$ equals the least element $\triangledown$ of
Part$(U)^{(d)}$. Thus we can apply Lemma~\ref{lemma 2.1} for $D(\mathcal{B})\cup\{t\}$ in
the lattice Part$(U)^{(d)}$, and we get that $D(\mathcal{B})\cup\{t\}$ is a
finite independent set, hence it generates a Boolean sublattice $\mathcal{B}%
^{^{\prime}}$ of Part$(U)^{(d)}$. Then $\mathcal{B}^{^{\prime}}$ is also a
Boolean sublattice in Part$(U)$. Since $\mathcal{B\subseteq B}^{^{\prime}}$
and $t\in\mathcal{B}^{^{\prime}}\setminus\mathcal{B}$, our statement holds.
\end{proof}

The following consequence is immediate:

\medskip

\begin{corollary}\label{cor. 2.3} Let $U\neq\emptyset$ be a finite set. Any
Boolean sublattice $\mathcal{B}$ of Part$(U)$ is included in a Boolean
sublattice $\mathcal{B}^{\ast}$ of Part$(U)$ containing $\triangle$ and
$\triangledown$.
\end{corollary}

Let $L$ be a lattice with a least element $0$. A nonempty subset
$D=\{d_{i}\mid i\in I\}\subseteq L$ is called a \emph{disjoint system}, if for
any $i,j\in I$, $i\neq j$, we have $d_{i}\wedge d_{j}=0$. $D$ is called
\emph{weakly independent} if $d_{i}\wedge\left(  \underset{j\in I\setminus
\{i\}}{\bigvee d_{j}}\right)  =0$ holds for all $i\in I$. Clearly, any
independent subset of $L$ is also weakly independent, and any weakly
independent subset of $L$ is a disjoint system, moreover, if $L$ is
distributive, then any disjoint system of $L$ is weakly independent (see
Horv\'{a}th and Radeleczki \cite{HorvathRadeleczki2012}). Now let $H$ be a nonempty subset of $L$. A
finite join of some elements in $H_{\wedge}$ is called an $\vee\wedge
$\emph{-element}, and their set is denoted by $H_{\vee\wedge}$. The following
result of S. Tamura \cite{Tamura1971} will be useful in our proofs:

\medskip

\begin{proposition}\label{prop. 2.4}(\cite[Corollary 2.]{Tamura1971}) Let $\langle
H\rangle$ be a sublattice generated by a nonempty subset $H$ of
a lattice $L$. Then the following are equivalent:

\noindent(i) $\langle H\rangle$ is distributive.

\noindent(ii) $(a\vee b)\wedge c=(a\wedge c)\vee(b\wedge c)$ holds for
any $a,b,c\in H_{\vee\wedge}$.
\end{proposition}

\noindent Let $\mathcal{B}$ be a finite Boolean lattice. Then $\mathcal{B}$ is
an atomistic lattice, and we will denote its set of atoms by $A(\mathcal{B})$.
By using the previous results, we prove

\medskip

\begin{lemma}\label{lemma 2.5} Let $L$ be a finite lattice and $D\subseteq
L\setminus\{0\}$ a disjoint system in $L$. Then the following assertions are equivalent:

\noindent(a) $D_{\vee}\cup\{0\}$ is a Boolean sublattice of $L$.

\noindent(b) $D_{\vee}\cup\{0\}$ is an atomistic lattice, and for any
$H\subseteq D_{\vee}$, $H\neq\emptyset$ and $d\in D$, $d\leq\vee H$ implies
$d\leq h$, for some $h\in H$.

\noindent(c) $(a\vee b)\wedge c=(a\wedge c)\vee(b\wedge c)$ holds for every
$a,b,c\in D_{\vee}\cup\{0\}$.

\noindent(d) $\langle D\rangle$ is semimodular and $D$ is weakly independent.

\noindent(e) $D$ is independent in $L$.
\end{lemma}

\begin{proof} (a)$\Rightarrow$(b). Since $B=D_{\vee}\cup\{0\}$ is a
finite Boolean lattice, it is atomistic, too. Let $a\in D_{\vee}$. If $a$ is
an atom of $B$, then we get $a\in D$. Hence $A(B)\subseteq D$. Observe that
for every $d\in D$ and $a_{1},...a_{k}\in A(B)$, $d=a_{1}\vee...\vee a_{k}$
implies $d=a_{1}=...=a_{k}$, as otherwise for any $a_{i}\neq d$ we would get
the contradiction $a_{i}=d\wedge a_{i}=0$, because $a_{i},d\in D$ and $D$ is a
disjoint system. This means that $A(B)=D$. Since $B$ is a Boolean lattice, its
atoms, i.e. the elements of $D$, are \emph{completely join-prime elements},
which means that for any $H\subseteq B$, $H\neq\emptyset$ and any $d\in
D=A(B)$, $d\leq\vee H$ implies $d\leq h$, for some $h\in H$. Thus (b) holds.

\noindent(b)$\Rightarrow$(c). According to (b), $B=D_{\vee}\cup\{0\}$ is an
atomistic lattice, so its join-irreducible elements are just its atoms lying
in $D$. Since by (b) these elements are completely join-prime in $B$, in view
of Reney \cite{Reny1952} result, $B$ is a distributive lattice. Hence (c) holds.

\noindent(c)$\Rightarrow$(d). Since $D$ is a disjoint system, we have
$D_{\wedge}=D\cup\{0\}$ and $D_{\vee\wedge}=D_{\vee}\cup\{0\}$. If (c) holds,
then $\langle D\rangle$ is a distributive lattice according to Proposition~\ref{prop. 2.4}. Because $D$ is a finite disjoint set in $\langle D\rangle$, in view of
\cite{HorvathRadeleczki2012} $D$ is weakly independent. Clearly, the lattice $\langle D\rangle$ is
semimodular, too.

\noindent(d)$\Rightarrow$(e) In view of \cite[Thm.4]{Gratzer2011} in a finite semimodular
lattice any weakly independent system of elements is independent. Thus $D$ is
independent in $L$.

\noindent(e)$\Rightarrow$(a). If $D\subseteq L$ is independent, then in view
of \cite{Gratzer2011}, the finite sublattice $\langle D\rangle$ of $L$ is a Boolean lattice,
and each element $x\in\langle D\rangle\setminus\{0\}$ is the image by the map
$\varphi\colon X\rightarrow\bigvee X$, of some subset $X\subseteq D$, i.e.
$x=\bigvee X\in D_{\vee}$. Hence $\langle D\rangle=$ $D_{\vee}\cup\{0\}$, thus
$D_{\vee}\cup\{0\}$ is a Boolean sublattice of $L$. 
\end{proof}

Next, we will examine the extensions of the Boolean sublattices of finite lattice $L$.

\medskip

\begin{lemma}\label{lemma 2.6} Let $\mathcal{B}$ be a Boolean sublattice of a
finite lattice $L$ such that $0,1\in B$. If $\mathcal{B}$ is strictly included
in a Boolean sublattice $\mathcal{D}$ of a $L$, then there exists an atom
$\alpha\in A(\mathcal{B})$ such that $\alpha$ is a proper join of some atoms
of $\mathcal{D}$ strictly less than $\alpha$. Moreover, there exists a proper
extension $\mathcal{B}^{\prime}$ of $\mathcal{B}$ included in $\mathcal{D}$,
such that its atoms are obtained by replacing in $A(\mathcal{B})$ this
$\alpha$ by two atoms $\delta,\delta^{\prime}$ of $\mathcal{B}^{\prime}$ with
$\delta\vee\delta^{\prime}=\alpha$.
\end{lemma}

\begin{proof} Because $\mathcal{D}$ is finite, any element of it is
a join of some atoms of $\mathcal{D}$. As $\mathcal{B\subseteq D}$, this holds
true for any atom of $\mathcal{B}$. Clearly, at least one atom $\delta$ of
$\mathcal{D}$ cannot belong to $\mathcal{B}$, otherwise $\mathcal{B}$ and
$\mathcal{D}$ would coincide. Thus $\delta\notin A(\mathcal{B})$. Now assume
that there is no any atom $\alpha$ of $\mathcal{B}$ that satisfies
$\delta<\alpha$. Then $\delta\wedge a=0$ for all $a\in A(\mathcal{B})$ imply
$\delta=\delta\wedge1=\delta\wedge\left(  \bigvee\{a\mid a\in A(\mathcal{B}%
)\}\right)  =\bigvee\{\delta\wedge a\mid a\in A(\mathcal{B})\}=0$, a
contradiction. Thus there exists at least one $\alpha\in A(\mathcal{B})$ with
$\delta<\alpha$. Then $\alpha$ is a proper join of some atoms of $\mathcal{D}$
strictly less than $\alpha$. Now let $\delta^{\prime}=\bigvee\{a\in
A(\mathcal{D})\mid a\leq\alpha$, $a\neq\delta\}$. Then clearly $\delta
\vee\delta^{\prime}=\alpha$ and $\delta\wedge\delta^{\prime}=0$. Let $B\prime$
be the Boolean sublattice of $\mathcal{D}$ generated by $\delta^{\prime}$ and
all the atoms of $\mathcal{D}$ not less than $\delta^{\prime}$ (among these
$\delta$). Clearly, $\mathcal{B\prime}$ extends $\mathcal{B}$ in $\mathcal{D}%
$, and $\delta$ and $\delta^{\prime}$ are atoms of $\mathcal{B\prime}$.
\end{proof}

We say that a proper extension $\mathcal{B\subset B}^{\sharp}$ of a Boolean
algebra $\mathcal{B}$ is an \emph{immediate (Boolean) extension} of it, if
there is no any Boolean algebra $\mathcal{D}$ with $\mathcal{B\subset D\subset
B}^{\sharp}$, where $\mathcal{\subset}$ means strict inclusion.

\medskip

\begin{corollary}\label{cor 2.7} Every\textbf{ }immediate Boolean extension
$\mathcal{B\subset B}^{\sharp}$ of a Boolean lattice $\mathcal{B}$ within a
finite lattice $L$ is obtained by replacing an atom $\alpha\in A(\mathcal{B})$
by two atoms $\delta,\delta^{\prime}$ of $\mathcal{B}^{\sharp}$ such that
$\alpha=\delta\vee\delta^{\prime}$.
\end{corollary}
\begin{proof} Let $\mathcal{B\subset B}^{\sharp}$ be an immediate Boolean
extension. In view of Lemma~\ref{lemma 2.6}, there is a Boolean algebra $\mathcal{B}%
^{\prime}$ with $\mathcal{B\subset B}^{\prime}\mathcal{\subseteq B}^{\sharp}$
which satisfies the required properties. Then by definition $\mathcal{B}%
^{\prime}=\mathcal{B}^{\sharp}$. Conversely, let $\mathcal{D}$ be a Boolean
extension of $\mathcal{B}$ obtained by replacing an atom $\alpha$ in
$A(\mathcal{B})$ by two atoms $\alpha_{1},\alpha_{2}$ of $\mathcal{D}$ such
that $\alpha=\alpha_{1}\vee\alpha_{2}$. Since $\mathcal{B\subset D}$, and
$\mathcal{D}$ is finite, there must exist a Boolean algebra $\mathcal{B}%
^{\sharp}\subseteq\mathcal{D}$ such that $\mathcal{B\subset B}^{\sharp}$ is an
immediate Boolean extension of $\mathcal{B}$. Then there exist two atoms
$\delta,\delta^{\prime}$ of $\mathcal{B}^{\sharp}$ such that $A(\mathcal{B}%
^{\sharp})=\left(  A(\mathcal{B})\setminus\{\alpha\}\right)  \cup
\{\delta,\delta^{\prime}\}$. Since the number of atoms of $\mathcal{B}%
^{\sharp}$ and $\mathcal{D}$ is the same (i.e. $\mid A(\mathcal{B})\mid+1$),
they must have the same size, therefore, $\mathcal{D}$ coincides with
$\mathcal{B}^{\sharp}$.
\end{proof}

\section{Boolean sublattices of Part(U) and their hypergraphs}\label{sec3}

\medskip

In this section, we assign to any disjoint system of a finite partition lattice
a hypergraph induced by the blocks of these partitions. The Boolean
sublattices of Part$(U)$ containing $\triangle$ will be characterized by a
condition imposed on the cycles of the hypergraphs corresponding to the atoms
of these sublattices.

\medskip

\noindent\ Let $U\neq\emptyset$ be a finite universe and $D$ be a disjoint
system in Part$(U)\setminus\{\triangle\}$. We define a hypergraph
$\mathcal{H}\left(  D\right)  =(U,\mathcal{E})$ as follows: the
\emph{vertices} (or the \emph{points }or\emph{ nodes}) of $\mathcal{H}\left(
D\right)  $ are the elements of the set $U$, and its \emph{edge set}
$\mathcal{E}$ (called also its \emph{set of blocks}) consists of the
non-singleton blocks of all partitions $\pi\in D$. By definition
$\mathcal{H}\left(  D\right)  $ can not contain loops or multiple edges,
moreover, two different edges $A_{1},A_{2}\in\mathcal{E}$ can have at most one
common vertex $v\in A_{1}\cap A_{2}$. Such an element $v$ will be called the
\emph{connection point of }$A_{1}$\emph{\ and }$A_{2}$. Indeed, if
$A_{1},A_{2}$ belong to the same partition $\alpha\in D$, then $A_{1}\cap
A_{2}=\emptyset$; if $A_{1},A_{2}$ belong to different partitions $\alpha
_{1},\alpha_{2}\in D$, then $\alpha_{1}\wedge\alpha_{2}=\triangle$ implies
$\mid A_{1}\cap A_{2}\mid\leq1$. These properties mean that
$\mathcal{H}\left(  D\right)  $ is a \emph{simple linear hypergraph. }(see
e.g. Dorfling and Henning \cite{DorflingHenning2014}) Observe also, that two different partitions
$\pi_{1},\pi_{2}\in D$, $\pi_{1}\neq\pi_{2}$ can have only common blocks with
a single element.

\medskip

\begin{definition}\label{def 3.1} The \emph{incidence graph} of a hypergraph
$\mathcal{H}\left(  D\right)  $ is a bipartite graph with vertex set
$U\cup\mathcal{E}$ in which $x\in U$ is adjacent to $B\in\mathcal{E}$ if and only if
$x\in B$.

\noindent(i) A \emph{walk} $\mathcal{P}$ in the hypergraph $\mathcal{H}\left(
D\right)  $ is a walk in the incidence graph starting and ending at points in
$U$. A \emph{path} is a walk in which no points or edges are repeating, except
that the starting and ending point may coincide.

\noindent(ii) If these two points are the same, then the path $\mathcal{P}$ is
called a \emph{cycle} $\mathcal{C}$ in $\mathcal{H}\left(  D\right)  $. In
fact, this is a sequence $v_{1},B_{1},v_{2},B_{2},v_{3},...,v_{n},B_{n},v_{1}$
, where $v_{1}\in B_{1}\cap B_{n},$ $v_{i}\in B_{i-1}\cap B_{i}$, $2\leq i\leq
n,$($n\geq3$), and the points $v_{1},...,v_{n}\in U$ and all its edges
$B_{1},...,B_{n}\in\mathcal{E}$ are different.
\end{definition}

Notice that every cycle contains the same number of points and edges.
This number is called the \emph{length} of the cycle, and it has to be at
least three, because our hypergraph is simple and linear. A hypergraph is
called \emph{connected}, if any pair of its vertices $a\neq b$ are connected
by a path. A connected hypergraph without cycles is called a \emph{hypertree}
(see \cite{cichacz2013cordial}).

\medskip

\begin{remark}\label{rem. 3.2} Let $\mathcal{B}$ be a Boolean sublattice of
Part$(U)$ containing $\triangle$, and $A(\mathcal{B})$ its set of atoms. Then
it is easy to see that $\triangledown\in\mathcal{B}$ if and only if the
hypergraph $\mathcal{H}\left(  A(\mathcal{B})\right)  $ corresponding to the
disjoint system $A(\mathcal{B})$ is connected.

Indeed, $\triangledown\in\mathcal{B}$ implies $\bigvee A(\mathcal{B}%
)=\triangledown$. Hence, for any $a,b\in U$, $a\neq b$, $(a,b)\in\alpha_{1}%
\vee...\vee\alpha_{n}$, for some $\alpha_{1},...,\alpha_{n}\in A(\mathcal{B}%
)$, because any principal partition $\pi_{\{a,b\}}$ is a compact element of
Part$(U)$. This means that there exists a path $v_{0},A_{1},v_{1}%
,A_{2},...,v_{n-1},A_{n}v_{n}$ with $v_{0}=a$, $v_{n}=b$ in $\mathcal{H}%
\left(  A(\mathcal{B})\right)  $ formed by the blocks $A_{1},...,A_{n}$ of
some partitions from $\{\alpha_{1},...,\alpha_{n}\}$, in other words, the
hypergraph $\mathcal{H}\left(  A(\mathcal{B})\right)  $ is connected.

Conversely, if $\mathcal{H}\left(  A(\mathcal{B})\right)  $ is connected, then
every pair $a,b\in U$, $a\neq b$ is connected by a path $v_{0},A_{1}%
,v_{1},A_{2},...,v_{n-1},A_{n}v_{n}$ with $v_{0}=a$, $v_{n}=b$ in
$\mathcal{H}\left(  A(\mathcal{B})\right)  $, where $A_{1},...,A_{n}$ are
blocks of some partitions $\alpha_{1},...,\alpha_{n}\in A(\mathcal{B})$. Then
$(a,b)\in\alpha_{1}\vee...\vee\alpha_{n}\leq\bigvee A(\mathcal{B})$, and this
implies $\triangledown=\bigvee A(\mathcal{B})\in\mathcal{B}$.
\end{remark}

\medskip

\begin{proposition}\label{prop. 3.3} Let $\mathcal{B}$ be a Boolean sublattice
of Part$(U)$ with $\triangle\in\mathcal{B}$, and $A(\mathcal{B})$ its set of
atoms. Then in each cycle $\mathcal{C}:=v_{1},B_{1},v_{2},B_{2},v_{3}%
,...,v_{n},B_{n},v_{1}$ in $\mathcal{H}\left(  A(\mathcal{B})\right)  $, whose
edges $B_{1},...,B_{n}$ are (non-singleton) blocks of partitions
$\alpha_{1},...,\alpha_{k}\in A(\mathcal{B})$, each partition $\alpha
_{1},...,\alpha_{k}$ participates with at least two non-singleton blocks.
\end{proposition}

\begin{proof} Assume\textit{\ }that\textit{\ }$\mathcal{C}%
:=v_{1},B_{1},v_{2},B_{2},v_{3},...,v_{n},B_{n},v_{1}$ is a cycle in
$\mathcal{H}\left(  A(\mathcal{B})\right)  $ such that that $B_{1}$ is a
(non-singleton) block of $\alpha_{1}\in A(\mathcal{B})$, and all the other
edges $B_{2},...,B_{n}$ are non-singleton blocks of some partitions
$\alpha_{2},...,\alpha_{k}\in A(\mathcal{B})$, different from $\alpha_{1}$.
Then the set $B_{2}\cup...\cup B_{n}$ has to be included in a block $A$ of the
partition $\alpha_{2}\vee...\vee\alpha_{k}$. Since $v_{2},v_{1}\in B_{2}%
\cup...\cup B_{n}\subseteq A$ and $v_{1},v_{2}\in B_{1}$, $v_{1}\neq v_{2}$,
we obtain $\mid B_{1}\cap A\mid\geq2$. However, this is a contradiction,
because by distributivity of $\mathcal{B}$, $\alpha_{1}\wedge\alpha
_{2}=\triangle,...,\alpha_{1}\wedge\alpha_{k}=\triangle$ imply $\alpha
_{1}\wedge(\alpha_{2}\vee...\vee\alpha_{k})=\triangle$, and this means that
any block $B_{i}$ of $\alpha_{1}$ and the block $A$ of $\alpha_{2}\vee
...\vee\alpha_{k}$ can have at most one common element.
\end{proof}

\medskip

\begin{corollary}\label{cor. 3.4} Let $\mathcal{B}$ be a Boolean sublattice of
Part$(U)$ with $\triangle\in\mathcal{B}$.

\noindent(i) If each atom of $\mathcal{B}$ is a principal partition, then the
hypergraph $\mathcal{H}\left(  A(\mathcal{B})\right)  $ does not contain a
cycle. If in addition $\triangledown\in\mathcal{B}$, then $\mathcal{H}\left(
A(\mathcal{B})\right)  $ is a hypertree.

\noindent(ii ) If a cycle $\mathcal{C}$ in $\mathcal{H}\left(  A(\mathcal{B}%
)\right)  $ contains at most three edges none of which is a non-singleton
block of an atom that has a non-singleton block among the other edges of
$\mathcal{C}$, then these edges belong to the same atom of $\mathcal{B}$.
\end{corollary}

\begin{proof} (i) If $\mathcal{C}$ would be cycle $\mathcal{C}$ in
$\mathcal{H}\left(  A(\mathcal{B})\right)  $, then in view of Proposition 3.3,
it would contain at least two non-singleton blocks of some atom $\pi_{A}$ of
it, which is impossible, because $\pi_{A}$ is a principal partition. As
$\mathcal{H}\left(  A(\mathcal{B})\right)  $ is a hypergraph without cycles,
by Remark 3.2 it is a hypertree whenever $\triangledown\in\mathcal{B}$.

\noindent(ii) By Proposition 3.3 this is clear if there are two edges with
this property in $\mathcal{C}$. Suppose that there are $3$ edges in
$\mathcal{C}$ with such a property. If they would be shared among more
elements of $A(\mathcal{B})$, then by our assumption, one of them would be a
block of an $\alpha\in A(\mathcal{B})$ that has no other non-singleton blocks
included in $\mathcal{C}$. However, this case is excluded by Proposition 3.3.
\end{proof}

\medskip

Let $\mathcal{D}\subseteq\ $Part$(U)\setminus\{\triangle\}$ be a disjoint
system, and consider now the corresponding hypergraph $\mathcal{H}\left(
\mathcal{D}\right)  =(U,\mathcal{E)}$.

\medskip

\noindent\textbf{Circuit} \textbf{Condition.} We say the hypergraph
$\mathcal{H}\left(  D\right)  $ \emph{satisfies the circuit condition}, if for
any subsets $X,Y\subseteq\mathcal{D}$ and $Z=X\cap Y$ and to any pair of paths
$\mathcal{P}_{1}:=$ $u,A_{1},v_{1},A_{2},...,v_{n-1},A_{n},v$ and
$\mathcal{P}_{2}:=$ $u,B_{1},w_{1},B_{2},...,w_{m-1},B_{m}v$ in $\mathcal{H}%
\left(  D\right)  $ with common endpoints $u$ and $v$ formed by the blocks
$A_{1},...,A_{n}$ of some partitions from $X$ and by the blocks $B_{1}%
,...,B_{m}$ of some partitions from $Y$, corresponds a path $u,C_{1}%
,u_{1},C_{2}...,u_{k-1},C_{k},v$ with $u\in C_{1}$ and $v\in C_{k}$, such that
$C_{1},...,C_{k}$ are blocks of some partitions from $Z$.

\bigskip

\begin{theorem}\label{Thm. 3.5} Let $U\neq\emptyset$ be a finite universe and
$\mathcal{D}\subseteq\ $Part$(U)\setminus\{\triangle\}$ a disjoint system.
Then $\mathcal{L}=\mathcal{D}_{\vee}\cup\{\triangle\}$ is a Boolean sublattice
of Part$(U)$ if and only if the hypergraph $\mathcal{H}\left(  \mathcal{D}%
\right)  $ satisfies the circuit condition.
\end{theorem}

\begin{proof} Let $\mathcal{L}=\mathcal{D}_{\vee}\cup\{\triangle\}$
be a Boolean lattice, $X=\{\alpha_{1},...,\alpha_{p}\}\subseteq\mathcal{D}$,
$Y=\{\beta_{1},...,\beta_{q}\}\subseteq\mathcal{D}$, $Z=X\cap Y$, and
$\mathcal{P}_{1}:=$ $u,A_{1},v_{1},A_{2},...,v_{n-1},A_{n},v$, $\mathcal{P}%
_{2}:=$ $u,B_{1},w_{1},B_{2},...,w_{m-1},B_{m}v$ two paths in $\mathcal{H}%
\left(  D\right)  $ with common endpoints $u$ and $v$, defined as in the
circuit condition.

Then the set $A_{1}\cup...\cup A_{n}$ is included in a block $E$ of the
partition $\alpha_{1}\vee...\vee\alpha_{p}$ and $B_{1}\cup...\cup B_{m}$ is
included in a block $F$ of $\beta_{1}\vee...\vee\beta_{q}$. Now, $u,v\in E\cap
F$ and $E\cap F$ is a block of $\left(  \alpha_{1}\vee...\vee\alpha
_{p}\right)  $ $\wedge$ $\left(  \beta_{1}\vee...\vee\beta_{q}\right)  $. Let
$Z:=\{\gamma_{1},...,\gamma_{r}\}$. As the set of the atoms of $\mathcal{L}$
is independent, $\{\alpha_{1},...,\alpha_{p}\}\cap\{\beta_{1},...,\beta
_{q}\}=\{\gamma_{1},...,\gamma_{r}\}$ yields $\left(  \alpha_{1}\vee
...\vee\alpha_{p}\right)  \wedge\left(  \beta_{1}\vee...\vee\beta_{q}\right)
=\gamma_{1}\vee...\vee\gamma_{r}$. Hence $(u,v)\in\gamma_{1}\vee...\vee
\gamma_{r}$. Therefore, there are different points $z_{0},z_{1},...,z_{k}\in
U$ with $u=z_{0}$, $v=z_{k}$, such that for each $1\leq i\leq k$ we have
$(z_{i-1},z_{i})\in\gamma_{j}$, for some $j\in\{1,...,r\}$. Then $z_{0}%
,z_{1}\in C_{1},...,z_{k-1},z_{k}\in C_{k}$, where $C_{1},...,C_{k}$ are
non-singleton blocks of some partitions in $Z=\{\gamma_{1},...,\gamma_{r}\}$.
Clearly, w.l.o.g. we may suppose that all these blocks are different. As
$u,C_{1},z_{1},C_{2},...,z_{k-1},C_{k},v$ is a path in $\mathcal{H}\left(
\mathcal{D}\right)  $, the circuit condition holds.

Conversely, suppose that $\mathcal{H}\mathcal{(D})$ satisfies the circuit
condition. In view of Lemma 2.5, to prove that $\mathcal{L}=D_{\vee}%
\cup\{\triangle\}$ is a Boolean lattice, it is enough to show that
$\mathcal{D}$ is independent, i.e. $\{\alpha_{1},...,\alpha_{p}\}\cap
\{\beta_{1},...,\beta_{q}\}=\{\gamma_{1},...,\gamma_{r}\}$ implies $\left(
\alpha_{1}\vee...\vee\alpha_{p}\right)  \wedge\left(  \beta_{1}\vee
...\vee\beta_{q}\right)  =\gamma_{1}\vee...\vee\gamma_{r}$, for any
$X=\{\alpha_{1},...,\alpha_{p}\}\subseteq\mathcal{D}$ and $Y=\{\beta
_{1},...,\beta_{q}\}\subseteq\mathcal{D}$. Clearly, $\gamma_{1}\vee
...\vee\gamma_{r}\leq\left(  \alpha_{1}\vee...\vee\alpha_{p}\right)
\wedge\left(  \beta_{1}\vee...\vee\beta_{q}\right)  $. In order to prove the
equality, choose any pair $(u,v)\in\left(  \alpha_{1}\vee...\vee\alpha
_{p}\right)  \wedge\left(  \beta_{1}\vee...\vee\beta_{q}\right)  $, $u\neq v$.
Then there are some elements $v_{0},...,v_{n},w_{0},...,w_{m}\in U$ with
$v_{0}=w_{0}=u$, and $v_{n}=w_{m}=v$ and such that $(v_{i-1},v_{i})\in
\alpha_{i}\in\{\alpha_{1},...,\alpha_{p}\}$, for each $1\leq i\leq n$, and
$(w_{k-1},w_{k})\in\beta_{k}\in\{\beta_{1},...,\beta_{q}\}$, for $1\leq k\leq
m$. Hence, there exist two paths $\mathcal{P}_{1}:=$ $u,A_{1},v_{1}%
,A_{2},...,v_{n-1},A_{n},v$ and $\mathcal{P}_{2}:=$ $u,B_{1},w_{1}%
,B_{2},...,w_{m-1},B_{m}v$ with common endpoints $u$ and $v$ as in the circuit
condition. Then there is a path $u,C_{1},u_{1},C_{2}...,u_{k-1},C_{k},v$ in
$\mathcal{H}\left(  D\right)  $ with $u\in C_{1}$ and $v\in C_{k}$ formed from
blocks $C_{1},...,C_{k}$ of some partitions from $X\cap Y=\{\gamma
_{1},...,\gamma_{r}\}$. Thus $(u,v)\in\gamma_{1}\vee...\vee\gamma_{r}$, and
hence $\left(  \alpha_{1}\vee...\vee\alpha_{p}\right)  \wedge\left(  \beta
_{1}\vee...\vee\beta_{q}\right)  =\gamma_{1}\vee...\vee\gamma_{r}$
holds.
\end{proof}

\medskip

\begin{remark}\label{rem. 3.6} Note that when we examine the fulfillment of the
circuit condition in $\mathcal{H}\left(  D\right)  $, it is enough to examine
such pairs $\mathcal{P}_{1}:=$ $u,A_{1},v_{1},A_{2},...,v_{n-1},A_{n},v$ and
$\mathcal{P}_{2}:=$ $u,B_{1},w_{1},B_{2},...,w_{m-1},B_{m}v$ of paths that
form cycles, i.e. have no common points except $u$ and $v$.

Indeed, let $A_{1},...,A_{n}$ be blocks of some partitions in $X\subseteq D$
and $B_{1},...,B_{m}$ blocks of partitions in $Y\subseteq D$. If
$\mathcal{P}_{1}$ and $\mathcal{P}_{2}$ do not form a cycle, then their
consecutive common points determine either some cycles, or some common subpaths
of them. As all edges of the common subpaths are blocks of some partitions
from $X\cap Y$, it suffices to check the fulfillment of the circuit condition for
these cycles.
\end{remark}

\medskip

\begin{theorem}\label{Thm. 3.7} Let $U\neq\emptyset$ be a finite set, and
$\mathcal{B}$ a Boolean sublattice of Part$(U)$ containing $\triangle$. If
$\alpha$ is an atom in $\mathcal{B}$, then the sublattice $\mathcal{L}$ of
Part$(U)$ generated by $\left(  A(\mathcal{B})\mathcal{\setminus\{\alpha
\}}\right)  \mathcal{\cup\{\pi}_{A}\mid A$ is a block of $\alpha\}$ is a
Boolean lattice if and only if the hypergraph $\mathcal{H}\left(
A(\mathcal{B})\right)  $ does not contain cycles including a block of $\alpha$.
\end{theorem}

\begin{proof} Observe that $\mathcal{D}=\left(  A(\mathcal{B}%
)\mathcal{\setminus\{\alpha\}}\right)  \mathcal{\cup\{\pi}_{A}\mid A$ is a
block of $\alpha\}$ is a disjoint system in Part$(U)$ and $\triangle
\in\mathcal{B\subseteq\langle}\mathcal{D}\mathcal{\rangle}$.

First, assume that $\mathcal{L}:=\mathcal{\langle}\mathcal{D}\mathcal{\rangle
}$ is a Boolean lattice. Since $\mathcal{D}_{\vee}\cup\{\triangle
\}\subseteq\mathcal{L}$ and $\mathcal{L}$ is distributive, Lemma~\ref{lemma 2.5} yields
that $\mathcal{D}_{\vee}\cup\{0\}$ is a sublattice of Part$(U)$. As
$\mathcal{D\subset D}_{\vee}\cup\{0\}$, we get $\mathcal{L}=\mathcal{\langle
}\mathcal{D}\mathcal{\rangle=\ }\mathcal{D}_{\vee}\cup\{0\}$, and
$A(\mathcal{L})=\mathcal{D}$. Assume that $\mathcal{C}$ is a cycle in
$\mathcal{H}\left(  A(\mathcal{B})\right)  $ containing some (non-singleton)
blocks $A_{1},...,A_{k}$ of $\alpha$. As all these edges are blocks of some
partitions from $\mathcal{D}=A(\mathcal{L})$, $\mathcal{C}$ is also a cycle in
the hypergraph $\mathcal{H}\left(  A(\mathcal{L})\right)  $. Since in
$\mathcal{C}$ any principal partition $\mathcal{\pi}_{A_{i}}\in A(\mathcal{L})$,
$1\leq i\leq k$ can be present with one non-singleton block only, this is a
contradiction to Proposition~\ref{prop. 3.3}.

\medskip

Conversely, assume that $\mathcal{H}\left(  A(\mathcal{B})\right)  $ has no
cycles including blocks of $\alpha$. As $\mathcal{D}$ is a disjoint set, by
Theorem~\ref{Thm. 3.5}, $\mathcal{L}=\mathcal{D}_{\vee}\cup\{\triangle\}$ is a Boolean
sublattice of Part$(U)$ whenever $\mathcal{H}\left(  \mathcal{D}\right)  $
satisfies the circuit condition. Now, let $u,A_{1},v_{1},A_{2},...,v_{n-1}%
,A_{n},v$ and $u,B_{1},w_{1},B_{2},...,w_{m-1},B_{m}v$ be two paths in
$\mathcal{H}\left(  D\right)  $ with common endpoints $u$ and $v$ such that
$A_{1},...,A_{n}$ and $B_{1},...,B_{m}$ are blocks of some partitions from
$X\subseteq\mathcal{D}$, respectively from $Y\subseteq\mathcal{D}$. By Remark~\ref{rem. 3.6}, it is enough to consider the case when $\mathcal{C}:=$ $u,A_{1}%
,...,v_{n-1},A_{n},v,B_{m},w_{m-1},...,B_{1},u$ is a cycle in $\mathcal{H}%
\left(  D\right)  $. As by assumption no block of $\alpha$ is included in
$\mathcal{C}$, we get $X,Y\subseteq A(\mathcal{B})\mathcal{\setminus
\{\alpha\}}$. Since $\mathcal{B}=A(\mathcal{B})_{\vee}\cup\{\triangle\}$ is a
Boolean lattice, by Theorem~\ref{Thm. 3.5}, $\mathcal{H}\left(  A(\mathcal{B})\right)  $
satisfies the circuit condition. Hence, there is a path $u,C_{1},u_{1}%
,...,u_{k-1},C_{k},v$ with $u\in C_{1}$, $v\in C_{k}$ in $\mathcal{H}\left(
A(\mathcal{B})\right)  $ such that $C_{1},...,C_{k}$ are blocks of some
partitions from $X\cap Y$. Since $X\cap Y\subseteq A(\mathcal{B}%
)\mathcal{\setminus\{\alpha\}\subseteq D}$, this means that $\mathcal{H}%
\left(  D\right)  $ also satisfies the circuit condition. Thus $\mathcal{L}%
=D_{\vee}\cup\{\triangle\}$ is a Boolean lattice.
\end{proof}

\medskip

\begin{proposition}\label{prop. 3.8} Let $U\neq\emptyset$ be a finite set and
$\mathcal{B}$ a Boolean sublattice of Part$(U)$ containing $\triangle$ and
$\bigtriangledown$.

\noindent(i) $\mathcal{H}\left(  A(\mathcal{B})\right)  $ is a hypertree if
and only if $\mathcal{B}$ can be extended to a Boolean sublattice
$\mathcal{B}^{\ast}$ of Part$(U)$ such that the atoms of $\mathcal{B}^{\ast}$
are all principal partitions $\mathcal{\pi}_{A}$, where $A$ is a block of an
atom $\alpha\in A(\mathcal{B})$. \ Moreover, $\mathcal{B}^{\ast}$ determines
on $U$ the same hypergraph as $\mathcal{B}$, i.e. $\mathcal{H}\left(
A(\mathcal{B}^{\ast})\right)  =\mathcal{H}\left(  A(\mathcal{B})\right)  $.

\noindent(ii) If $\mathcal{H}\left(  A(\mathcal{B})\right)  $ is a hypertree,
and $\mathcal{B}^{\ast}$ is the extension of $\mathcal{B}$ as in (i), then
$\mathcal{B}^{\ast}=\mathcal{D}(U,\mathcal{I})$, where $\mathcal{I}$ is a
semi-interval system formed by the blocks of the partitions $\pi\in
\mathcal{B}^{\ast}$and $\emptyset$.
\end{proposition}

\begin{proof} (i) Assume that $\mathcal{H}\left(
A(\mathcal{B})\right)  $ is a hypertree. Observe that replacing an atom
$\alpha\in A(\mathcal{B)}$ by the principal partitions $\mathcal{\pi}_{A}$ of
its non-singleton blocks, the edges of $\mathcal{H}\left(  A(\mathcal{B}%
)\right)  $ are not changed, i.e. the hypergraph remains the same. Since
$A(\mathcal{B})$ is finite, we can replace each atom $\alpha\in A(\mathcal{B}%
)$ by principal partitions in finite steps. In view of Theorem~\ref{Thm. 3.7}, at each
step $\mathcal{B}$ is extended to a Boolean sublattice of Part$(U)$, so that
at the end we obtain a Boolean sublattice $\mathcal{B}^{\ast}$. According to
this construction, $A(\mathcal{B}^{\ast})$ is formed by all principal
partitions $\mathcal{\pi}_{A}$, where $A$ is a block of an atom $\alpha\in
A(\mathcal{B})$, and clearly, $\mathcal{H}\left(  A(\mathcal{B}^{\ast
})\right)  =\mathcal{H}\left(  A(\mathcal{B})\right)  $ also holds.

\smallskip
Conversely, suppose that $\mathcal{B}$ can be extended to a Boolean sublattice
$\mathcal{B}^{\ast}$ of Part$(U)$ having as atoms all the partitions
$\mathcal{\pi}_{A}$, where $A$ is a non-singleton block of an atom of
$\mathcal{B}$. As $\triangledown\in\mathcal{B}$, the hypergraph $\mathcal{H}%
\left(  A(\mathcal{B})\right)  $ is connected. $\mathcal{H}\left(
A(\mathcal{B})\right)  $ has no cycles: Indeed, for any cycle $\mathcal{C}$:
$v_{1},B_{1},v_{2},B_{2},v_{3},...,v_{n},B_{n},v_{1}$ in $\mathcal{H}\left(
A(\mathcal{B})\right)  $, $B_{1},B_{2},...,B_{n}$ are blocks of some atoms in
$\mathcal{B}$, and they are all different. As now the principal partitions
$\mathcal{\pi}_{B_{i}}$, $1\leq i\leq n$ are different atoms in $\mathcal{B}%
^{\ast}$, this is a contradiction to Proposition~\ref{prop. 3.3}, because $\mathcal{C}$ is
also a cycle in $\mathcal{H}\left(  A(\mathcal{B}^{\ast})\right)  $, however
each $\mathcal{\pi}_{B_{i}}$ has only one non-singleton block in $\mathcal{C}%
$. Therefore, $\mathcal{H}\left(  A(\mathcal{B})\right)  $ is a hypertree.

\medskip

\noindent(ii) Let $\mathcal{H}\left(  A(\mathcal{B})\right)  $ be a hypertree,
and $\mathcal{B}^{\ast}$ the Boolean extension of $\mathcal{B}$ constructed in
(i). We prove that adding $\emptyset$ to the set of blocks of the partitions
from $\mathcal{B}^{\ast}$ we obtain a closure system $\mathcal{I\subseteq}%
\wp(U)$ satisfying condition I$_{0}$. Indeed, $U\in\mathcal{I}$, since
$\triangledown\in\mathcal{B\subseteq B}^{\ast}$. As $\triangle\in
\mathcal{B\subseteq B}^{\ast}$, each $\{x\}$, $x\in U$ belongs to
$\mathcal{I}$. Take any $A_{k}\in\mathcal{I}$, $k\in K$. If $\underset{k\in
K}{\bigcap}A_{k}$ is a singleton or $\emptyset$, then $\underset{k\in
K}{\bigcap}A_{k}\in\mathcal{I}$. Otherwise, each $A_{k}$ is a non-singleton
block of a partition $\pi_{A_{k}}\in\mathcal{B}^{\ast}$, $k\in K$, hence
$\underset{k\in K}{\bigcap}A_{k}$ is a block of $\underset{k\in K}{\bigwedge
}\pi_{A_{k}}\in\mathcal{B}^{\ast}$, i.e. $\underset{k\in K}{\bigcap}A_{k}%
\in\mathcal{I}$. Next, we prove $\mathcal{B}^{\ast}=\mathcal{D}(U,\mathcal{I}%
)$. Indeed, as each block of a partition $\pi\in\mathcal{B}^{\ast}$ belongs to
$\mathcal{I}$, we have $\mathcal{B}^{\ast}\subseteq\mathcal{D}(U,\mathcal{I}%
)$. Let $\pi=\{A_{i}$ $\mid i\in I\}\in\mathcal{D}(U,\mathcal{I})$. Then by
Remark~\ref{rem. 1.1}, $\pi=\bigvee\{\pi_{A_{i}}\mid i\in I\}$ holds in Part$(U)$,
moreover, $\pi_{A_{i}}\in\mathcal{B}^{\ast}$ by our definition. As
$\mathcal{B}^{\ast}$ is a finite sublattice of Part$(U)$, we get $\pi
\in\mathcal{B}^{\ast}$, whence $\mathcal{D}(U,\mathcal{I})\subseteq
\mathcal{B}^{\ast}$. Since $\mathcal{D}(U,\mathcal{I})=\mathcal{B}^{\ast}$ is
a (complete) semimodular sublattice of Part$(U)$, in view of Proposition~\ref{prop. 1.2}
$\mathcal{I}$ also satisfies I$_{1}$, thus it is a semi-interval system.
\end{proof}

\medskip

\begin{corollary}\label{cor. 3.9} (i) Let $\mathcal{H}_{0}$ be a hypertree
corresponding to a Boolean sublattice $\mathcal{B}_{0}$ of Part$(U)$
containing $\triangle$ and $\triangledown$. Then there exists a semi-interval
system $\mathcal{I\subseteq\wp}(U)$ such that its decomposition lattice
$\mathcal{D}(U,\mathcal{I)}$ is a maximal Boolean sublattice of Part$(U)$ with
the property $\mathcal{H}\left(  A(\mathcal{B}_{0})\right)  =\mathcal{H}_{0}$.

(ii) Let $\mathcal{I}$ be a semi-interval system on $U$ such that
$\mathcal{D}(U,\mathcal{I)}$ is an atomistic lattice. Then $\mathcal{D}%
(U,\mathcal{I})$ is a Boolean sublattice of Part$(U)$ if and only if the atoms
$\mathcal{A}$ of $\mathcal{D}(U,\mathcal{I)}$ define a hypertree
$\mathcal{H(A})$ on $U$.
\end{corollary}
\begin{proof} (i) In view of Proposition~\ref{prop. 3.8}, there exists a semi-interval
system $\mathcal{I\subseteq\wp}(U)$, such that $\mathcal{B}_{0}\mathcal{\ }$is
included in the Boolean lattice $\mathcal{D}(U,\mathcal{I})$ and
$\mathcal{H(A})=\mathcal{H}_{0}$, where $\mathcal{A}$ denotes the set of atoms
of $\mathcal{D}(U,\mathcal{I})$. Assume that there is a Boolean sublattice
$\mathcal{B}$ of Part$(U)$ with $\mathcal{D}(U,\mathcal{I})\subseteq
\mathcal{B}$ and $\mathcal{H}\left(  A(\mathcal{B})\right)  =\mathcal{H}_{0}$.
We show that each atom $\alpha\in\mathcal{B}$ belongs to $\mathcal{D}%
(U,\mathcal{I})$. As $\mathcal{B}$ is an atomistic lattice, this is enough,
because it implies $\mathcal{B\subseteq D}(U,\mathcal{I})$. Indeed, each block
$A_{i}$, $i\in I$ of $\alpha$ is an edge of $\mathcal{H}_{0}$, hence it is
also a block of an atom of $\mathcal{D}(U,\mathcal{I})$ which is a principal
partition $\pi_{A_{i}}$, with $A\in\mathcal{I}$. Since $\mathcal{D}%
(U,\mathcal{I})$ is a finite sublattice of Part$(U)$, we obtain $\alpha
=\bigvee\{\pi_{A_{i}}\mid i\in I\}\in\mathcal{D}(U,\mathcal{I})$, completing
our proof.

\noindent(ii) Assume that $\mathcal{D}(U,\mathcal{I})$ is an atomistic
lattice. As $\mathcal{D}(U,\mathcal{I})$ satisfies I$_{0}$, its atoms are
principal partitions $\mathcal{\pi}_{A}$, with $A\in\mathcal{I}$, $\mid
A\mid\geq2$. Hence in view of Proposition~\ref{prop. 3.8}(i), $\mathcal{H(A})$ is a
hypertree, whenever $\mathcal{D}(U,\mathcal{I})$ is a Boolean sublattice of
Part$(U)$. Conversely, if $\mathcal{H}\left(  \mathcal{A})\right)  $ is a
hypertree, then by Remark~\ref{rem. 3.6}, the circuit condition is trivially satisfied by
$\mathcal{H(A})$, because $\mathcal{H(A})$ has no cycles at all. Since
$\mathcal{I}$ is a semi-interval system, the set $\mathcal{A}$ of the atoms of
$\mathcal{D}(U,\mathcal{I)}$ is a disjoint system in Part$(U)$. As
$\mathcal{D}(U,\mathcal{I)}$ is an atomistic lattice, $\mathcal{D}%
(U,\mathcal{I)}=\mathcal{A}_{\vee}\cup\{0\}$. By applying Theorem~\ref{Thm. 3.5} with
$\mathcal{D}:=\mathcal{A}$, we get that $\mathcal{D}(U,\mathcal{I)}$ is a
Boolean sublattice of Part$(U)$.
\end{proof}

\medskip

\section{Maximal Boolean sublattices of a finite partition
lattice}\label{sec4}

\medskip

In Section 1 we have seen that the largest size Boolean sublattices of
Part$(U)$ correspond to hypergraphs which are trees on the node set $U$,
however, there are examples of Boolean sublattices of Part$(U)$, which are
maximal with respect to inclusion, although their size is not the largest (see
e.g. Example~\ref{ex. 4.1}). Therefore, in this section we examine some properties of
Boolean sublattices of Part$(U)$ that cannot be extended to larger Boolean
sublattices, as well as some characteristics of Boolean extensions. As the
main result of this section we prove that there exist maximal Boolean
sublattices in Part$(U)$ for all possible sizes.

\medskip

\begin{example}\label{ex. 4.1} Let\textbf{ }$U=\{a,b,c,d\}$. Then the largest
size Boolean sublattices of Part$(U)$ are $3$ dimensional, i.e. they have $8$
elements. Let us define the partitions $\pi_{1}=\{\{a,b\},\{c,d\}\}$ and
$\pi_{2}=\{\{a,c\},\{b,d\}\}$. Then $\pi_{1}\vee\pi_{2}=\triangledown$ and
$\pi_{1}\wedge\pi_{2}=\triangle$, hence $\mathcal{B}=\{\triangle,\pi_{1}%
,\pi_{2},\triangledown\}$ is Boolean sublattice of Part$(U)$ with $4$
elements, and $\pi_{1},\pi_{2}$ are its atoms. Observe that $\mathcal{B}$ has
no a proper Boolean extension. In fact, by Corollary~\ref{cor 2.7}, an immediate Boolean
extension of $\mathcal{B}$ can only be obtained by splitting $\pi_{1}$ or
$\pi_{2}$ into two (new) atoms, but in both cases a cycle containing a
non-singleton block of a principal partition is formed, which contradicts
Proposition~\ref{prop. 3.3}.
\end{example}

\medskip

In all what follows, let $U$ be a nonempty finite set, and $\mathcal{B}$ a
Boolean sublattice of Part$(U)$ such that $\triangle,\nabla\in\mathcal{B}$.
Because $\mathcal{B}$ is finite, it is an atomistic lattice.

\medskip

\begin{lemma}\label{lemma 4.2} Let $\mathcal{B}$ and $\mathcal{D}$ Boolean
sublattices of Part$(U)$ such that $\triangle\in\mathcal{B\subseteq D}$. If
the hypergraph $\mathcal{H}\left(  A(\mathcal{B})\right)  $ contains a cycle
$\mathcal{C}$ then $\mathcal{H}\left(  A(\mathcal{D})\right)  $ also contains
a cycle $\mathcal{C}^{\ast}$.
\end{lemma}

\begin{proof} We may suppose $\mathcal{B\neq D}$. Let $\mathcal{C}%
$: $v_{1},B_{1},v_{2},...,v_{n},B_{n},v_{n+1}$ with $v_{n+1}=v_{1}$, be a
cycle in $\mathcal{H}\left(  A(\mathcal{B})\right)  $, i.e. let $B_{1}%
,...,B_{n}$ be different blocks of some atoms $\alpha_{1},...,\alpha_{m}$ of
$\mathcal{B}$. If one of these blocks belongs to an atom which is also an atom
in $\mathcal{D}$, then this block remains also an edge in the hypergraph
$\mathcal{H}\left(  A(\mathcal{D})\right)  $. If $B_{k}$ ($k\in\{1,...n\}$)
belongs to some $\alpha\notin A(\mathcal{D)}$, then in view of Lemma~\ref{lemma 2.6}
there are atoms $\delta_{1},...,\delta_{t}\in A(\mathcal{D)}$ such that
$\alpha=\delta_{1}\vee...\vee\delta_{t}$. As $(v_{k},v_{k+1})\in B_{k}%
^{2}\subseteq\alpha$, there exists a sequence $z_{1},...z_{p}\in U$, with
$z_{1}=v_{k}$, $z_{p}=v_{k+1}$ such that $(z_{i},z_{i+1})\in\delta_{j}$, for
some $j\in\{1,...t\}$, for all $1\leq i\leq p-1$. This means that each
$\{z_{i},z_{i+1}\}$ is contained in a block $D_{i}$ of some partition from
$\delta_{1},...,\delta_{t}$. W.l.o.g we may assume that the sequence
$z_{1},D_{1},z_{2},...,z_{p-1},D_{p-1},z_{p}$ has no repeating points and
blocks, i.e. it is a path in $\mathcal{H}\left(  A(\mathcal{D})\right)  $.
Therefore, any edge of $\mathcal{C}$ is either an edge in $\mathcal{H}\left(
A(\mathcal{D})\right)  $, or it can be replaced in $\mathcal{H}\left(
A(\mathcal{D})\right)  $ by a path. If these new paths have no common points,
then inserting them in $\mathcal{C}$ we get a cycle in $\mathcal{H}\left(
A(\mathcal{D})\right)  $. Otherwise, we get a closed walk $\mathcal{W}$, and
two consecutive repetitions of a point determine either a cycle or a subpath
of this $\mathcal{W}$. Since not all the points of these new paths can be
common to all of them, there must exist a cycle $\mathcal{C}^{\ast}$ defined
by two repeating points of $\mathcal{W}$, and by its construction,
$\mathcal{C}^{\ast}$ is included in $\mathcal{H}\left(  A(\mathcal{D})\right) $
\end{proof}

\medskip

Corollary~\ref{cor 2.7} points out that any immediate Boolean extension of a Boolean
sublattice $\mathcal{B}$ in Part$(U)$ is constructed by replacing an atom
$\alpha\in\mathcal{B}$ by two "new" disjoint partitions $\delta,\delta
^{\prime}\in\ $Part$(U)$ with $\delta\vee\delta^{\prime}=\alpha$. Such a
simple situation is considered in the following way.

\medskip

\begin{lemma}\label{lemma 4.3} Let $U\neq\emptyset$ be a finite set,
$\mathcal{B}$ a Boolean sublattice of Part$(U)$ with $\triangle$
$,\triangledown\in\mathcal{B}$, and $\mathcal{\pi}_{A}$ an atom in
$\mathcal{B}$ such that the set $A\subseteq U$ contains at least three
elements. If $A=D\cup E$, such that$\mid D\mid$,$\mid E\mid\geq2$ and $\mid
D\cap E\mid=1$, then the sublattice $\mathcal{L}$ of Part$(U)$ generated by
$\left(  A(\mathcal{B})\mathcal{\setminus\{\pi}_{A}\mathcal{\}}\right)
\mathcal{\cup\{\pi}_{D},\mathcal{\pi}_{E}\}$ is a Boolean lattice strictly
containing $\mathcal{B}$.
\end{lemma}

\begin{proof} This is clear, if $\mathcal{B=\{}\triangle$
$,\triangledown\}$. Assume that $\mid\mathcal{B\mid\geq}4$. Since
$\mathcal{\pi}_{A}=\mathcal{\pi}_{D}\vee\mathcal{\pi}_{E}$, and $\mathcal{\pi
}_{D},\mathcal{\pi}_{E}\notin\mathcal{B}$, the lattice $\mathcal{L}$ properly
contains $\mathcal{B}$. Because $\mathcal{B}$ is a Boolean sublattice of
Part$(U)$, $A(\mathcal{B})$ and $A(\mathcal{B})\mathcal{\setminus\{\pi}%
_{A}\mathcal{\}}$ are independent sets in Part$(U)$. Now, let $m:=\bigvee
\left(  A(\mathcal{B})\mathcal{\setminus\{\pi}_{A}\mathcal{\}}\right)  $ and
$p:=m\vee\mathcal{\pi}_{D}$. Observe that $\mathcal{\pi}_{D}\wedge
\mathcal{\pi}_{E}=\triangle$, $\mathcal{\pi}_{A}\wedge m=\triangle$. As all
principal partitions are modular elements in Part$(U)$, we have $\left(
\mathcal{\pi}_{D},m\right)  ,(\mathcal{\pi}_{E},p),\left(  m,\mathcal{\pi}%
_{A}\right)  \in M$.

Since $\mathcal{\pi}_{D}\wedge m=\triangle$ and $\left(\mathcal{\pi}
_{D},m\right)  \in M$, by using Lemma~\ref{lemma 2.1} we get that $\left(  A(\mathcal{B}%
)\mathcal{\setminus\{\pi}_{A}\mathcal{\}}\right)  \mathcal{\cup\{\pi}_{D}\}$
is an independent set, and 
\begin{equation}
\bigvee\left(\left(  A(\mathcal{B}
)\mathcal{\setminus\{\pi}_{A}\mathcal{\}}\right)  \mathcal{\cup\{\pi}%
_{D}\}\right)  =m\vee\mathcal{\pi}_{D}=p. \tag{3}
\end{equation}

\noindent Because $\left(  m,\mathcal{\pi
}_{A}\right)  \in M$ and $\mathcal{\pi}_{D},\mathcal{\pi}_{E}\leq\mathcal{\pi
}_{A}$ we obtain

\smallskip

$ p\wedge\mathcal{\pi}_{E}=(\mathcal{\pi}_{D}\vee m)\wedge\left(\mathcal{\pi
}_{A}\wedge\mathcal{\pi}_{E}\right)  =
[(\mathcal{\pi}_{D}\vee m)\wedge
\mathcal{\pi}_{A}]\wedge\mathcal{\pi}_{E}=
 [\mathcal{\pi}_{D}\vee
(m\wedge\mathcal{\pi}_{A})]\wedge\mathcal{\pi}_{E}=
(\mathcal{\pi}_{D}\vee\triangle)\wedge\mathcal{\pi}_{E}=\triangle$

\smallskip
Since $(\mathcal{\pi}_{E},p)\in M$ also holds, by applying Lemma~\ref{lemma 2.1}
again, we get that

\centering{$ \mathcal{H}=\left(  A(\mathcal{B})\mathcal{\setminus\{\pi
}_{A}\mathcal{\}}\right)  \mathcal{\cup\{\pi}_{D},\mathcal{\pi}_{E}\} $}

is an independent system in Part$(U)$. In view of Lemma 2.5, this implies that
$\mathcal{H}_{\vee}\cup\{0\}$ is a Boolean sublattice of Part$(U)$. Then the
sublattice $\mathcal{L}$ of Part$(U)$ generated by $\mathcal{H}$ coincides with
$\mathcal{H}_{\vee}\cup\{0\}$. 
\end{proof}

\medskip

By using the previous lemmas, we are able now to answer the question of weather
a Boolean sublattice of Part$(U)$ is contained in a Boolean sublattice with
a largest possible size.

\bigskip

\begin{theorem}\label{Thm. 4.4} Let $\mathcal{B}$ be a Boolean sublattice of
Part$(U)$ with $\triangle,\nabla\in\mathcal{B}$. Then $\mathcal{B}$ can be
extended to a Boolean sublattice of Part$(U)$ with a largest size if and only
if the hypergraph $\mathcal{H}\left(  A(\mathcal{B})\right)  $ assigned to its
atoms is a hypertree.
\end{theorem}

\begin{proof} Assume that $\mathcal{H}\left(  A(\mathcal{B}%
)\right)  $ is a hypertree. Then in view of Proposition~\ref{prop. 3.8}, $\mathcal{B}$ can
be extended to a Boolean sublattice $\mathcal{B}^{\ast}$ of Part$(U)$ having
as atoms principal partitions $\mathcal{\pi}_{A}$, where $A\subseteq U$ is a
finite set with $\mid A\mid\geq2$ and $\mathcal{H}\left(  A(\mathcal{B}^{\ast
})\right)  $ is the same as $\mathcal{H}\left(  A(\mathcal{B})\right)  $. If
$A=\{a_{1},a_{2},...,a_{n}\}$, with $n\geq3$, then in view of Lemma~\ref{lemma 4.3},
$\mathcal{\pi}_{A}$ can be replaced by partitions $\mathcal{\pi
}_{\{a_{1},a_{2}\}}$ and $\mathcal{\pi}_{\{a_{2},...,a_{n}\}}$ such that the
obtained new set $\left(  A(\mathcal{B})\mathcal{\setminus\{\pi}_{A}\mathcal{\}}%
\right)  \mathcal{\cup\{\pi}_{\{a_{1},a_{2}\}},\mathcal{\pi}_{\{a_{2}%
,...,a_{n}\}}\}$ of atoms defines a proper Boolean extension of $\mathcal{B}$.
Since $\mathcal{B}^{\ast}$is finite, all its atoms can be replaced in finite
steps by principal partitions having non-singleton blocks with two elements,
and these partitions are atoms of a Boolean sublattice $\widehat{\mathcal{B}}\subseteq
\ $Part$(U)$ that extends $\mathcal{B}^{\ast}$. As the edges of $\mathcal{H}%
\left(  A(\widehat{\mathcal{B}})\right)  $ have two elements, $\mathcal{H}%
\left(  A(\widehat{\mathcal{B}})\right)  $ is a (simple) graph. Since the atoms
of $\widehat{\mathcal{B}}$ are principal partitions and $\triangle
,\triangledown\in\mathcal{B\subseteq}$ $\widehat{\mathcal{B}}$, in view of
Corollary~\ref{cor. 3.4}(i), this graph has no cycles and it is connected, i.e. it is a
tree. In view of Proposition~\ref{prop. 1.4}, $\widehat{\mathcal{B}}$ is a Boolean
sublattice of Part$(U)$ with a largest size, and $\mathcal{B\subseteq}$
$\widehat{\mathcal{B}}$.

Conversely, assume that $\mathcal{B}$ can be extended to a Boolean sublattice
$\mathcal{D}$\ of Part$(U)$ of largest size. Then the hypergraph
$\mathcal{H}\left(  A(\mathcal{D})\right)  $ must be a tree, according to the
proof of Proposition~\ref{prop. 1.4}. Since by Lemma~\ref{lemma 4.2} cycles do not disappear under
Boolean extensions, the hypergraph $\mathcal{H}\left(  A(\mathcal{B})\right)
$ cannot contain cycles. As $\nabla\in\mathcal{B}$, in view of Remark~\ref{rem. 3.2},
$\mathcal{H}\left(  A(\mathcal{B})\right)  $ is connected, thus it is a
hypertree.
\end{proof}

By applying now Corollary 3.9 (ii) we obtain:

\medskip

\begin{corollary}\label{cor. 4.5} Let $\mathcal{I}$ be a semi-interval system
on $U$ such that $\mathcal{D}(U,\mathcal{I)}$ is a Boolean lattice. Then
$\mathcal{D}(U,\mathcal{I)}$ can be extended to a Boolean sublattice of
Part$(U)$ with a largest size.
\end{corollary}

One of the simplest ways to decompose a multi-block atom $\alpha$ of a Boolean
sublattice $\mathcal{B}$ of Part$(U)$ in the form $\alpha=\alpha_{1}\vee
\alpha_{2}$ with $\alpha_{1}\wedge\alpha_{2}=\triangle$, it is to divide its
non-singleton blocks $A_{1},...,A_{n}$ into two groups $A_{1},...,A_{k}$ and
$A_{k+1},...,A_{n}$ so that $\left(  \pi_{A_{1}}\vee...\vee\pi_{A_{k}}\right)
\wedge\left(  \pi_{A_{k+1}}\vee...\vee\pi_{A_{n}}\right)  =\triangle$, and
then define the partitions $\alpha_{1}=\pi_{A_{1}}\vee...\vee\pi_{A_{k}}$,
$\alpha_{2}=\pi_{A_{k+1}}\vee...\vee\pi_{A_{n}}$. The question we would like
to find at least a partial answer to, is under what conditions does the
sublattice of Part$(U)$ generated by $\mathcal{D}=\left(  A(\mathcal{B}%
)\mathcal{\setminus\{\alpha\}}\right)  \mathcal{\cup\{}\alpha_{1},\alpha
_{2}\}$ is a Boolean lattice?

\medskip

We will show that in the case when a non-singleton block $A_{1}$ of $\alpha$ is not
contained in any cycle of $\mathcal{H}\left(  A(\mathcal{B})\right)  $, then
setting $\alpha_{1}$ equal to the principal partition $\pi_{A_{1}}$ and
$\alpha_{2}:=\pi_{A_{2}}\vee...\vee\pi_{A_{n}}$, this pair $\alpha_{1}%
,\alpha_{2}$ will do the job. \medskip

Clearly, $\alpha_{1}\wedge\alpha_{2}=\triangle$ and it is easy to check that
$\mathcal{D}=\left(  A(\mathcal{B})\mathcal{\setminus\{\alpha\}}\right)
\mathcal{\cup\{}\alpha_{1},\alpha_{2}\}$ is a disjoint system. We prove that
$\mathcal{H}\left(  D\right)  $ satisfies the circuit condition. 

Indeed, let
$\mathcal{P}_{1}$: $u,A_{1},v_{1},A_{2},...,v_{n-1},A_{n},v$ and
$\mathcal{P}_{2}$: $u,B_{1},w_{1},B_{2},...,w_{m-1},B_{m}v$ be two paths in
$\mathcal{H}\left(  D\right)  $ with common endpoints $u$ and $v$, and such
that $A_{1},...,A_{n}$ and $B_{1},...,B_{m}$ are blocks of some partitions
from $X\subseteq\mathcal{D}$, respectively from $Y\subseteq\mathcal{D}$. In
view of Remark~\ref{rem. 3.6}, we may suppose that they form a cycle in $\mathcal{H}%
\left(  D\right)  $. Since all above blocks belong to some partition
$\pi\in A(\mathcal{B})$, and the circuit condition holds in $\mathcal{H}%
\left(  A(\mathcal{B})\right)  $, there exists a path $\mathcal{P}:$
$u,C_{1},u_{1},...,u_{k-1},C_{k},v$ with $u\in C_{1}$, $v\in C_{k}$ in
$\mathcal{H}\left(  A(\mathcal{B})\right)  $ such that $C_{1},...,C_{k}$ are
blocks of some partitions $\gamma_{1},...,\gamma_{r}$ in $\mathcal{E}%
\cap\mathcal{F}$, where $\mathcal{E}$ respectively $\mathcal{F}$, are obtained
from $X$, respectively $Y$, by replacing any occurrence of $\alpha_{1}$and
$\alpha_{2}$ by $\alpha$. If $\alpha_{1},\alpha_{2}\notin X\cap Y$, then
clearly, $\gamma_{1},...,\gamma_{r}\in X\cap Y$ also holds.

\smallskip

Now, assume that some $\pi\in\mathcal{\{}\alpha_{1},\alpha_{2}\}$ appears in
$X\cap Y$, i.e. both $\mathcal{P}_{1}$ and $\mathcal{P}_{2}$ contain some
blocks of $\alpha$. Then $\mathcal{P}$ also contains some block of $\alpha$.
Observe that the block $A_{1}$ is not contained in any of them. Indeed, as
$\mathcal{P}$ together with $\mathcal{P}_{1}$, and $\mathcal{P}$ together with
$\mathcal{P}_{2}$ form closed walks, and $A_{1}$ is not included in any cycle,
$A_{1}$ would be a common edge for all of them, whenever it is contained in one
of them. However, this is a contradiction, because in such a case
$\mathcal{P}_{1}$ and $\mathcal{P}_{2}$ would not form a cycle. Thus each
block of $\alpha$ that appears in one of these paths must be a block of
$\alpha_{2}$, and hence $\mathcal{P}$ also belongs to $\mathcal{H}\left(
D\right)  $. As now any element of $X\cup Y$ which is not in $A(\mathcal{B}%
)\mathcal{\setminus\{\alpha\}}$ equals to $\alpha_{2}$, by replacing $\alpha$
by $\alpha_{2}$ in the set $\{\gamma_{1},...,\gamma_{r}\}$, $\gamma
_{1},...,\gamma_{r}\in X\cap Y$ holds again. Thus the circuit condition is
satisfied by $\mathcal{H}\left(  D\right)  $, therefore, $\mathcal{L}=D_{\vee
}\cup\{\triangle\}$ is a Boolean lattice, according to Theorem~\ref{Thm. 3.5}.

\medskip
Hence, in view of Corollary~\ref{cor 2.7}, we proved the following:

\medskip

\begin{proposition}\label{prop. 4.6} Let $\mathcal{B}$ be a Boolean sublattice
of Part$(U)$ with $\triangle,\nabla\in\mathcal{B}$. If an atom $\alpha$ of
$\mathcal{B}$ has a non-singleton block $A_{1}$ that is not contained in any
cycle of the hypergraph $\mathcal{H}\left(  A(\mathcal{B})\right)  $, and
$\alpha$ has some other non-singleton blocks $A_{2},...,A_{n}$ also, then the
sublattice of Part$(U)$ induced by $\left(  A(\mathcal{B})\mathcal{\setminus
\{\alpha\}}\right)  \mathcal{\cup\{}\alpha_{1},\alpha_{2}\}$, where
$\alpha_{1}=\pi_{A_{1}}$ and $\alpha_{2}=\pi_{A_{2}}\vee...\vee\pi_{A_{n}}$,
is an immediate Boolean extension of $\mathcal{B}$.
\end{proposition}

\medskip

\begin{corollary}\label{cor. 4.7} Let $\mathcal{B}$ be a maximal Boolean
sublattice of Part$(U)$. If $A$ is a non-singleton block of an atom $\alpha$
of $\mathcal{B}$ which is not contained in any cycle of the hypergraph
$\mathcal{H}\left(  A(\mathcal{B})\right)  $, then $\alpha=\pi_{A}$ and $A$
has two elements.
\end{corollary}
\begin{proof} Since $\mathcal{B}$ is a maximal Boolean sublattice
of Part$(U)$, in view of Corollary~\ref{cor. 2.3} we have $\triangle$,$\triangledown
\in\mathcal{B}$. Now, as a common consequence of Proposition~\ref{prop. 4.6} and Lemma~\ref{lemma 4.3}
we obtain the required result. 
\end{proof}

\medskip

\begin{remark}\label{rem. 4.8} We already know several properties of a Boolean
sublattice $\mathcal{B}$ of Part$(U)$ which is maximal but not of the largest
possible size. First, $\mathcal{B}$ contains $\bigtriangleup$ and
$\triangledown$, and consequently the hypergraph $\mathcal{H}\left(
A(\mathcal{B})\right)  $ is connected, and secondly, $\mathcal{H}\left(
A(\mathcal{B})\right)  $ contains at least one cycle. If $\alpha$ is an atom
of $\mathcal{B}$, then either $\alpha$ is a principal partition with a two
element block (and hence, it is an atom in Part$(U)$ also), or each
non-singleton block of $\alpha$ is included in some cycle of $\mathcal{H}%
\left(  A(\mathcal{B})\right)  $, and any such cycle contains at least two
non-singleton blocks of $\alpha$.
\end{remark}

\medskip

\begin{example}\label{ex. 4.9} If $\mid U\mid=5$, then the maximal Boolean
sublattices of Part$(U)$ have $8$ or $16$ elements. Indeed, by Proposition~\ref{prop. 1.4}, the largest size Boolean sublattices of Part$(U)$ have $2^{4}=16$ elements.
\smallskip

\noindent Let $\mathcal{B}=\{\triangle,\alpha,\beta,\triangledown\}$ be a
Boolean sublattice of Part$(U)$ with size $4$. Then $\alpha,\beta$ are
complements each of other. Assume by contradiction that $\mathcal{B}$ is a
maximal Boolean sublattice of Part$(U)$ and let $\mathcal{H}$ be the
hypergraph assigned to the set $\{\alpha,\beta\}$. Then, in view of Remark~\ref{rem. 4.8}, $\mathcal{H}$ is connected and has a cycle $\mathcal{C}$ that contains at
least two non-singleton blocks of both partitions, so it covers at least $4$
points $a,b,c,d$ of $U$. 

Let $e$ be fifth element of $U$. Then $\{e\}$ can be
a block of at most one of them, otherwise $\mathcal{H}$ would not be
connected. The case when $\alpha$ has a $3$-element block $A=\{a,b,c\}$ and
$\beta$ has a $3$-element block $B$, can also be excluded, since then $\mid
A\cap B\mid=1$ would imply that the second non-singleton block $\{d,e\}$ of
$\alpha$ is included in $B$, contradicting $\alpha\wedge\beta=\triangle$. Thus
w.l.o.g. we may assume that $\alpha$ has as blocks $\{a,b\}$, $\{c,d\}$,
$\{e\}$ and $\beta$ has two blocks $B_{1},B_{2}$, where $\mid B_{1}\mid=2$ and
$\mid B_{2}\mid=3$, and the non-singleton blocks of them form the cycle
$\mathcal{C}$. By symmetry, we can suppose $B_{1}=\{a,c\}$, $B_{2}=\{b,d,e\}$.
Then $\beta$ is a join of the partitions $\ \beta_{1}=\pi_{\{a,c\}}\vee
\pi_{\{b,d\}}$ and $\beta_{2}=\pi_{\{d,e\}}$. It is easy to check that
$\{\alpha,\beta_{1},\beta_{2}\}$ is an independent set in Part$(U)$. Hence the
Boolean lattice $\mathcal{B}^{\ast}$ induced by $\{\alpha,\beta_{1},\beta
_{2}\}$ is a proper extension of $\mathcal{B}$, contrary to our assumption.

Now consider the hypergraph $\mathcal{H}^{\ast\text{ }}$corresponding to
$\{\alpha,\beta_{1},\beta_{2}\}$. Since $\mathcal{H}^{\ast\text{ }}$contains
the cycle $\mathcal{C}$, $\mathcal{B}^{\ast}$ can not be extended to a largest
size Boolean sublattice of Part$(U)$ with $16$ elements. As $\mathcal{B}%
^{\ast}$ has $8$ elements, $\mathcal{B}^{\ast}$ is a maximal Boolean
sublattice of Part$(U)$.
\end{example}

\medskip

\begin{lemma}\label{lemma 4.10} Let $\mid U\mid=n\geq3$.

\noindent(i) If $n=3k$ or $n=3k+1$, where $k\in\mathbb{N}$, then there exists
a maximal Boolean sublattice of Part$(U)$ with $4$ elements.

\noindent(ii) If $n=3k$ with $k\geq2$, $k\in\mathbb{N}$, then there exists a
maximal Boolean sublattice of Part$(U)$ with $8$ elements.
\end{lemma}

\begin{proof} (i) is clear if $U$ has three elements, since then
Part$(U)\cong M_{3}$. For $\mid U\mid=4$, (i) also holds according to Example~\ref{ex. 4.1}. Thus we may assume that $\mid U\mid=n\geq6$. We will distinguish two
cases: \smallskip

\noindent Case (a). Let $n=3k$, where $k\in\mathbb{N}$, $k\geq2$. We divide
$U$ into three sets $A=\{a_{1},...,a_{k}\}$, $B=\{b_{1},...,b_{k}\}$ and
$C=\{c_{1},...,c_{k}\}$. Let $\alpha$ be the partition having as blocks $A$
and $E_{1}=\{b_{1},c_{1}\},...,E_{k}=\{b_{k},c_{k}\}$, and $\beta$ the
partition with the blocks $C$ and $F_{1}=\{a_{1},b_{1}\},...,F_{k}%
=\{a_{k},b_{k}\}$, as in Fig. 4.1.

\begin{figure}[H]
\centering
\begin{adjustbox}{width=\textwidth}
\begin{tikzpicture}
\coordinate (a1) at (0, 4);
\coordinate (ai) at (2, 4);
\coordinate (ai+1) at (3, 4);
\coordinate (ak) at (5, 4);
\coordinate (b1) at (0, 2);
\coordinate (bi) at (2, 2);
\coordinate (bi+1) at (3, 2);
\coordinate (bk) at (5, 2);
\coordinate (c1) at (0, 0);
\coordinate (ci) at (2, 0);
\coordinate (ci+1) at (3, 0);
\coordinate (ck) at (5, 0);
\coordinate (a2i) at (7, 4);
\coordinate (a2i+1) at (8, 4);
\coordinate (b2i) at (7, 2);
\coordinate (b2i+1) at (8, 2);
\coordinate (c2i) at (7, 0);
\coordinate (c2i+1) at (8, 0);
\coordinate (as) at (10, 4);
\coordinate (at) at (12, 4);
\coordinate (av) at (14, 4);
\coordinate (bs) at (10, 2);
\coordinate (bt) at (12, 2);
\coordinate (bv) at (14, 2);
\coordinate (cs) at (10, 0);
\coordinate (ct) at (12, 0);
\coordinate (cv) at (14, 0);
\coordinate (bici) at (7.5, -1.2);
\coordinate (btct) at (12.5, -1.2);
\draw (a1) -- node[midway, above = -0.5mm] {$\alpha$} (ai) -- node[midway, below = -0.5mm] {$\alpha$} (ai+1) -- node[midway, above = -0.5mm] {$\alpha$} (ak);
\draw (a2i) -- node[midway, circle, draw, inner sep=2pt, below = 0.5mm] {$\alpha_1$} (a2i+1);
\draw[decorate, decoration={snake, amplitude=0.5mm, segment length = 1mm, pre length=0pt,
post length=0pt}]
(as) -- node[midway, circle, draw, inner sep=2pt, above = 1mm] {$\alpha_1$} (at);
\draw[decorate, decoration={snake, amplitude=0.5mm, segment length = 1mm, pre length=0pt,
post length=0pt}]
(at) -- node[midway, circle, draw, inner sep=2pt, above = 1mm] {$\alpha_2$} (av);
\draw[color=red] (c1) -- node[color=red, midway, below = -0.5mm] {$\beta$} (ci) -- node[color=red, midway, above = -0.5mm] {$\beta$} (ci+1) -- node[color=red, midway, below = -0.5mm] {$\beta$} (ck);
\draw[color=red] (c2i) -- node[color=red, midway, above = -0.5mm] {$\beta$} (c2i+1);
\draw[color=red] (cs) -- node[color=red, midway, below = -0.5mm] {$\beta$} (ct) -- node[color=red, midway, below = -0.5mm] {$\beta$} (cv);
\draw (c1) -- node[midway, right = -0.5mm] {$\alpha$} node[midway, right = 6mm] {$\cdots$} (b1);
\draw (ci) -- node[midway, left = -0.5mm] {$\alpha$} (bi);
\draw (ci+1) -- node[midway, right = -0.5mm] {$\alpha$} node[midway, right = 6mm] {$\cdots$} (bi+1);
\draw (ck) -- node[midway, left = -0.5mm] {$\alpha$} (bk);
\draw (c2i) -- node[midway, circle, draw, inner sep=2pt, left = 0.5mm] {$\alpha_1$} (b2i);
\draw (c2i+1) -- node[midway, circle, draw, inner sep=2pt, right = 0.5mm] {$\alpha_1$}(b2i+1);
\draw (cs) -- node[midway, circle, draw, inner sep=2pt, right = 0.5mm] {$\alpha_1$} (bs);
\draw (ct) -- node[midway, circle, draw, inner sep=2pt, left = 0.75mm] {$\alpha_1$} node[midway, circle, draw, inner sep=2pt, right = 0.75mm] {$\alpha_2$} node[midway, above = 0.2mm] {\Large\textbf{?}} (bt);
\draw (cv) -- node[midway, circle, draw, inner sep=2pt, left = 0.5mm] {$\alpha_2$} (bv);
\draw[color=red] (a1) -- node[color=red, midway, right = -0.5mm] {$\beta$} node[color=red, midway, right = 6mm] {$\cdots$} (b1);
\draw[color=red] (ai) -- node[color=red, midway, left = -0.5mm] {$\beta$} (bi);
\draw[color=red] (ai+1) -- node[color=red, midway, right = -0.5mm] {$\beta$} node[color=red, midway, right = 6mm] {$\cdots$} (bi+1);
\draw[color=red] (ak) -- node[color=red, midway, left = -0.5mm] {$\beta$} (bk);
\draw[color=red] (a2i) -- node[color=red, midway, left = -0.5mm] {$\beta$} (b2i);
\draw[color=red] (a2i+1) -- node[color=red, midway, right = -0.5mm] {$\beta$} (b2i+1);
\draw[color=red] (as) -- node[color=red, midway, left = -0.5mm] {$\beta$} (bs);
\draw[color=red] (at) -- node[color=red, midway, right = -0.5mm] {$\beta$} (bt);
\draw[color=red] (av) -- node[color=red, midway, left = -0.5mm] {$\beta$} (bv);
\draw node[above = 1mm of a1] {$a_1$};
\draw node[above = 1mm of ai] {$a_i$};
\draw node[above = 1mm of ai+1] {$a_{i+1}$};
\draw node[above = 1mm of ak] {$a_k$};
\draw node[above = 1mm of a2i] {$a_i$};
\draw node[above = 1mm of a2i+1] {$a_{i+1}$};
\draw node[above = 1mm of as] {$a_s$};
\draw node[above = 1mm of at] {$a_t$};
\draw node[above = 1mm of av] {$a_v$};
\draw node[below = 1mm of c1] {$c_1$};
\draw node[below = 1mm of ci] {$c_i$};
\draw node[below = 1mm of ci+1] {$c_{i+1}$};
\draw node[below = 1mm of ck] {$c_k$};
\draw node[below = 1mm of c2i] {$c_i$};
\draw node[below = 1mm of c2i+1] {$c_{i+1}$};
\draw node[below = 1mm of cs] {$c_s$};
\draw node[below = 1mm of ct] {$c_t$};
\draw node[below = 1mm of cv] {$c_v$};
\draw node[left = 1mm of b1] {$b_1$};
\draw node[left = 1mm of bi] {$b_i$};
\draw node[right = 1mm of bi+1] {$b_{i+1}$};
\draw node[right = 1mm of bk] {$b_k$};
\draw node[left = 1mm of b2i] {$b_i$};
\draw node[right = 1mm of b2i+1] {$b_{i+1}$};
\draw node[left = 1mm of bs] {$b_s$};
\draw node[left = 1mm of bt] {$b_t$};
\draw node[right = 1mm of bv] {$b_v$};
\draw (bici) node {$(b_i, c_i) \in \alpha_1, 1 \le i \le k$};
\draw (btct) node {$(b_t, c_t) \in \alpha_1 \wedge \alpha_2\text{ }?$};
\draw[fill=black] (a1) circle [radius=2pt];
\draw[fill=black] (ai) circle [radius=2pt];
\draw[fill=black] (ai+1) circle [radius=2pt];
\draw[fill=black] (ak) circle [radius=2pt];
\draw[fill=black] (b1) circle [radius=2pt];
\draw[fill=black] (bi) circle [radius=2pt];
\draw[fill=black] (bi+1) circle [radius=2pt];
\draw[fill=black] (bk) circle [radius=2pt];
\draw[fill=black] (c1) circle [radius=2pt];
\draw[fill=black] (ci) circle [radius=2pt];
\draw[fill=black] (ci+1) circle [radius=2pt];
\draw[fill=black] (ck) circle [radius=2pt];
\draw[fill=black] (a2i) circle [radius=2pt];
\draw[fill=black] (a2i+1) circle [radius=2pt];
\draw[fill=black] (b2i) circle [radius=2pt];
\draw[fill=black] (b2i+1) circle [radius=2pt];
\draw[fill=black] (c2i) circle [radius=2pt];
\draw[fill=black] (c2i+1) circle [radius=2pt];
\draw[fill=black] (as) circle [radius=2pt];
\draw[fill=black] (at) circle [radius=2pt];
\draw[fill=black] (av) circle [radius=2pt];
\draw[fill=black] (bs) circle [radius=2pt];
\draw[fill=black] (bt) circle [radius=2pt];
\draw[fill=black] (bv) circle [radius=2pt];
\draw[fill=black] (cs) circle [radius=2pt];
\draw[fill=black] (ct) circle [radius=2pt];
\draw[fill=black] (cv) circle [radius=2pt];
\draw[rounded corners, blue] ([xshift=-3mm,yshift=7mm]a1) rectangle ([xshift=3mm,yshift=-5mm]ak);
\draw[rounded corners, green] ([xshift=-7mm,yshift=7mm]b1) rectangle ([xshift=7mm,yshift=-3mm]bk);
\draw[rounded corners, red] ([xshift=-3mm,yshift=7mm]c1) rectangle ([xshift=3mm,yshift=-7mm]ck);
\draw[rounded corners, blue] ([xshift=-3mm,yshift=9mm]as) rectangle ([xshift=3mm,yshift=-3mm]at);
\draw[rounded corners, blue] ([xshift=-3mm,yshift=10mm]at) rectangle ([xshift=3mm,yshift=-4mm]av);
\draw[blue] node[above left = 5mm of a1] {$A$};
\draw[green] node[above left = 8mm of b1] {$B$};
\draw[red] node[above left = 5mm of c1] {$C$};
\draw[blue] node[above left = 5mm of as] {$A_1$};
\draw[blue] node[above right = 5mm of av] {$A_2$};
\end{tikzpicture}
\end{adjustbox}
\caption{Case (a)}%
\label{fig:casea}%
\end{figure}

Since $\alpha\wedge\beta=\triangle$ and $\alpha\vee\beta=\triangledown$,
$\mathcal{B=\{\triangle},\alpha,\beta,\triangledown\}$ is a Boolean sublattice
of Part$(U)$ with $4$ elements and $\alpha,\beta\in A(\mathcal{B})$. We show
that $\mathcal{B}$ has no proper Boolean extension in Part$(U)$. Indeed, in
view of Corollary~\ref{cor 2.7}, any immediate Boolean extension $\mathcal{B}%
^{\#}\subseteq\ $Part$(U)$ of $\mathcal{B}$ is obtained by replacing $\alpha$
or $\beta$ by the join of two atoms of $\mathcal{B}^{\#}$. W.l.o.g. we may
assume that $\alpha=\alpha_{1}\vee\alpha_{2}$, $\alpha_{1}\wedge\alpha
_{2}=\triangle$, for some $\alpha_{1},\alpha_{2}\in A(\mathcal{B}^{\#})$. Then
each non-singleton block of $\alpha$ must be a union of some non-singleton
blocks of $\alpha_{1}$ and $\alpha_{2}$. As the blocks $E_{1},...,E_{k}$ have
two elements, for them this union is trivial, i.e. some of them should be
blocks of $\alpha_{1}$, and the rest of them should be blocks of $\alpha_{2}$.

First, we show that $A$ can not be a block of $\alpha_{1}$ or $\alpha_{2}$.
Due to the symmetry of our construction, it is enough to only deal with the
case when $A$ is a block of $\alpha_{1}$. Then for each $i\in\{1,...,k-1\}$
the cycle $\mathcal{C}$ determined by the points $a_{i},b_{i},c_{i}%
,c_{i+1},b_{i+1},a_{i+1}$, contains only three edges that are not blocks of
$\beta$, namely $E_{i}=\{b_{i},c_{i}\}$, $E_{i+1}=\{b_{i+1},c_{i+1}\}$, and
$A$ that connects $a_{i}$ with $a_{i+1}$. As $A$ is a block of $\alpha_{1}\in
A(\mathcal{B}^{\#})$ and $\mathcal{C}$ is included in the hypergraph
$\mathcal{H}(A(\mathcal{B}^{\#}))$, in view of Corollary~\ref{cor. 3.4}(ii) $E_{i}$ and
$E_{i+1}$should be also blocks of $\alpha_{1}$. This means that all
$E_{1},...,E_{k}$ are blocks of $\alpha_{1}$, and hence we obtain $\alpha
_{1}=\alpha$. However, this is a contradiction, since it yields $\alpha_{2}=\triangle$.

Now assume that $A$ is the union of some overlapping blocks of $\alpha_{1}$
and $\alpha_{2}$. Then there exist at least two subsets $A_{1},A_{2}\subseteq
A$ with $A_{1}\cap A_{2}\neq\emptyset$, $A_{1}\setminus A_{2}\neq\emptyset$,
$A_{2}\setminus A_{1}\neq\emptyset$, and such that $A_{1}$ is a block of
$\alpha_{1}$ and $A_{2}$ is a block of $\alpha_{2}$. Let $a_{s}\in
A_{1}\setminus A_{2}$, $a_{t}\in A_{1}\cap A_{2}$ and $a_{v}\in A_{2}\setminus
A_{1}$, where $s,t,v\in\{1,...,k\}$. W.l.o.g. we may assume $s<t<v$. Let us
consider the cycles

\smallskip

$\mathcal{C}_{1}\colon$ $a_{s},A_{1},a_{t},F_{t},b_{t},E_{t},c_{t}%
,C,c_{s},E_{s},b_{s},F_{s},a_{s}$ and

$\mathcal{C}_{2}\colon$ $a_{t},A_{2},a_{v},F_{v},b_{v},E_{v},c_{v}%
,C,c_{t},E_{t},b_{t},F_{t},a_{t}$,

\smallskip

\noindent on Fig. 4.1. Since all the blocks of $\mathcal{C}_{1}$ and
$\mathcal{C}_{2}$ belong to the partitions $\alpha_{1},\alpha_{2}$ and $\beta$
which are atoms in $\mathcal{B}^{\#}$, $\mathcal{C}_{1}$ and $\mathcal{C}_{2}$
are included in the hypergraph $\mathcal{H}(A(\mathcal{B}^{\#}))$. Since in
$\mathcal{C}_{1},$ only $A_{1},E_{t},E_{s}$ are not blocks of $\beta$ and
$A_{1}$ is a block of $\alpha_{1}$, in view of Corollary~\ref{cor. 3.4}(ii), $E_{t}%
,E_{s}$ also should be blocks of $\alpha_{1}$. Similarly, since in
$\mathcal{C}_{2}$ only the blocks $A_{2},E_{v},E_{t}$ are not included in
$\beta$ but $A_{2}$ is a block of $\alpha_{2}$, we get that $E_{v},E_{t}$
also should be blocks of $\alpha_{2}$. Thus $E_{t}=\{b_{t},c_{t}\}$ must be a
common block of $\alpha_{1}$ and $\alpha_{2}$, which is a contradiction to
$\alpha_{1}\wedge\alpha_{2}=\triangle$.
\medskip

\noindent Case (b). Let $n=3k+1$, with $k\in\mathbb{N}$, $k\geq2$. Now we
include all the elements of $U$ into the subsets $A=\{a_{1},...,a_{k},d\}$,
$B=\{b_{1},...,b_{k}\}$, and $C=\{c_{1},...,c_{k},d\}$ and we denote
$a_{k+1}=c_{k+1}:=d$. We define the partitions $\alpha,\beta\in\ $Part$(U)$
such that the blocks of $\alpha$ are $A$ and $E_{1}=\{b_{1},c_{1}%
\},...,E_{k}=\{b_{k},c_{k}\}$, and the blocks of $\beta$ are $C$ and
$F_{1}=\{a_{1},b_{1}\},...,F_{k}=\{a_{k},b_{k}\}$, as in Fig. 4.2.
\begin{figure}[H]
\centering
\begin{adjustbox}{width=\textwidth}
\begin{tikzpicture}
\coordinate (a1) at (0, 4);
\coordinate (ai) at (2, 4);
\coordinate (ai+1) at (3, 4);
\coordinate (ak) at (5, 4);
\coordinate (b1) at (0, 2);
\coordinate (bi) at (2, 2);
\coordinate (bi+1) at (3, 2);
\coordinate (bk) at (5, 2);
\coordinate (d) at (7, 2);
\coordinate (c1) at (0, 0);
\coordinate (ci) at (2, 0);
\coordinate (ci+1) at (3, 0);
\coordinate (ck) at (5, 0);
\coordinate (as) at (10, 4);
\coordinate (at) at (12, 4);
\coordinate (bs) at (10, 2);
\coordinate (bt) at (12, 2);
\coordinate (bv) at (15, 2);
\coordinate (cs) at (10, 0);
\coordinate (ct) at (12, 0);
\coordinate (btct) at (12.5, -1.2);
\draw (a1) -- node[midway, above = -0.5mm] {$\alpha$} (ai) -- node[midway, below = -0.5mm] {$\alpha$} (ai+1) -- node[midway, above = -0.5mm] {$\alpha$} (ak);
\draw[decorate, decoration={snake, amplitude=0.5mm, segment length = 1mm, pre length=0pt,
post length=0pt}]
(as) -- node[midway, circle, draw, inner sep=2pt, above = 1mm] {$\alpha_1$} (at);
\draw[color=red] (c1) -- node[color=red, midway, below = -0.5mm] {$\beta$} (ci) -- node[color=red, midway, above = -0.5mm] {$\beta$} (ci+1) -- node[color=red, midway, below = -0.5mm] {$\beta$} (ck);
\draw[color=red] (cs) -- node[color=red, midway, below = -0.5mm] {$\beta$} (ct);
\draw (c1) -- node[midway, right = -0.5mm] {$\alpha$} node[midway, right = 6mm] {$\cdots$} (b1);
\draw (ci) -- node[midway, left = -0.5mm] {$\alpha$} (bi);
\draw (ci+1) -- node[midway, right = -0.5mm] {$\alpha$} node[midway, right = 6mm] {$\cdots$} (bi+1);
\draw (ck) -- node[midway, left = -0.5mm] {$\alpha$} (bk);
\draw (cs) -- node[midway, circle, draw, inner sep=2pt, right = 0.5mm] {$\alpha_1$} (bs);
\draw (ct) -- node[midway, circle, draw, inner sep=2pt, left = 0.75mm] {$\alpha_1$} node[midway, circle, draw, inner sep=2pt, right = 0.75mm] {$\alpha_2$} node[midway, above = 0.2mm] {\Large\textbf{?}} (bt);
\draw[color=red] (a1) -- node[color=red, midway, right = -0.5mm] {$\beta$} node[color=red, midway, right = 6mm] {$\cdots$} (b1);
\draw[color=red] (ai) -- node[color=red, midway, left = -0.5mm] {$\beta$} (bi);
\draw[color=red] (ai+1) -- node[color=red, midway, right = -0.5mm] {$\beta$} node[color=red, midway, right = 6mm] {$\cdots$} (bi+1);
\draw[color=red] (ak) -- node[color=red, midway, left = -0.5mm] {$\beta$} (bk);
\draw[color=red] (as) -- node[color=red, midway, left = -0.5mm] {$\beta$} (bs);
\draw[color=red] (at) -- node[color=red, midway, right = -0.5mm] {$\beta$} (bt);
\draw[color=red] (ck) -- node[color=red, midway, right = 0.5mm] {$\beta$} (d);
\draw (ak) -- node[midway, above = 0.5mm] {$\alpha$} (d);
\draw[decorate, decoration={snake, amplitude=0.5mm, segment length = 1mm, pre length=0pt,
post length=0pt}]
(at) -- node[midway, circle, draw, inner sep=2pt, above = 2mm] {$\alpha_2$} (bv);
\draw[color=red, decorate, decoration={snake, amplitude=0.5mm, segment length = 1mm, pre length=0pt,
post length=0pt}] (ct) -- node[color=red, midway, below = -0.5mm] {$\beta$} (bv);
\draw node[above = 1mm of a1] {$a_1$};
\draw node[above = 1mm of ai] {$a_i$};
\draw node[above = 1mm of ai+1] {$a_{i+1}$};
\draw node[above = 1mm of ak] {$a_k$};
\draw node[above = 1mm of as] {$a_s$};
\draw node[above = 1mm of at] {$a_t$};
\draw node[left = 1mm of b1] {$b_1$};
\draw node[left = 1mm of bi] {$b_i$};
\draw node[right = 1mm of bi+1] {$b_{i+1}$};
\draw node[left = 1mm of bk] {$b_k$};
\draw node[above = 1mm of d] {$d$};
\draw node[left = 1mm of bs] {$b_s$};
\draw node[left = 1mm of bt] {$b_t$};
\draw node[right = 1mm of bv] {$d = a_v = b_v$};
\draw node[below = 1mm of c1] {$c_1$};
\draw node[below = 1mm of ci] {$c_i$};
\draw node[below = 1mm of ci+1] {$c_{i+1}$};
\draw node[below = 1mm of ck] {$c_k$};
\draw node[below = 1mm of cs] {$c_s$};
\draw node[below = 1mm of ct] {$c_t$};
\draw (btct) node {$(b_t, c_t) \in \alpha_1 \wedge \alpha_2\text{ }?$};
\draw[fill=black] (a1) circle [radius=2pt];
\draw[fill=black] (ai) circle [radius=2pt];
\draw[fill=black] (ai+1) circle [radius=2pt];
\draw[fill=black] (ak) circle [radius=2pt];
\draw[fill=black] (b1) circle [radius=2pt];
\draw[fill=black] (bi) circle [radius=2pt];
\draw[fill=black] (bi+1) circle [radius=2pt];
\draw[fill=black] (bk) circle [radius=2pt];
\draw[fill=black] (d) circle [radius=2pt];
\draw[fill=black] (c1) circle [radius=2pt];
\draw[fill=black] (ci) circle [radius=2pt];
\draw[fill=black] (ci+1) circle [radius=2pt];
\draw[fill=black] (ck) circle [radius=2pt];
\draw[fill=black] (as) circle [radius=2pt];
\draw[fill=black] (at) circle [radius=2pt];
\draw[fill=black] (bs) circle [radius=2pt];
\draw[fill=black] (bt) circle [radius=2pt];
\draw[fill=black] (bv) circle [radius=2pt];
\draw[fill=black] (cs) circle [radius=2pt];
\draw[fill=black] (ct) circle [radius=2pt];
\draw[rounded corners, green] ([xshift=-7mm,yshift=7mm]b1) rectangle ([xshift=7mm,yshift=-3mm]bk);
\draw[rounded corners, blue] ([xshift=-3mm,yshift=9mm]as) rectangle ([xshift=3mm,yshift=-3mm]at);
\begin{scope}[rotate around={-33.69:(at)}]
\draw[rounded corners, blue]
([xshift=-5mm,yshift=9mm]at)
rectangle
([xshift=4.5mm,yshift=-4mm]bv);
\end{scope}
\draw[blue, rounded corners=4pt]
([xshift=-5mm,yshift=6mm]a1)
-- ([xshift=5mm,yshift=6mm]ak)
-- ([xshift=9mm,yshift=2mm]ak)
-- ([xshift=6mm,yshift=4mm]d)
-- ([xshift=-1mm,yshift=-4mm]d)
-- ([xshift=1mm,yshift=-5mm]ak)
-- ([xshift=-5mm,yshift=-5mm]a1)
-- cycle;
\draw[red, rounded corners=4pt]
([xshift=-5mm,yshift=6mm]c1)
-- ([xshift=-1mm,yshift=5mm]ck)
-- ([xshift=3mm,yshift=12mm]d)
-- ([xshift=10mm,yshift=4mm]d)
-- ([xshift=1mm,yshift=-8mm]ck)
-- ([xshift=-5mm,yshift=-8mm]c1)
-- cycle;
\draw[blue] node[above left = 8mm of a1] {$A$};
\draw[green] node[above left = 8mm of b1] {$B$};
\draw[red] node[above left = 5mm of c1] {$C$};
\draw[blue] node[above left = 5mm of as] {$A_1$};
\draw[blue] node[above right = 8mm of bv] {$A_2$};
\end{tikzpicture}
\end{adjustbox}
\caption{Case (b)}%
\label{fig:caseb}%
\end{figure}

Clearly, $\alpha\wedge\beta=\triangle$ and $\alpha\vee\beta=\triangledown$.
Supposing that the Boolean lattice $\mathcal{B=\{\triangle},\alpha
,\beta,\triangledown\}$ has a proper Boolean extension $\mathcal{B}^{\#}$ in
Part$(U)$, we get, as in case (a), that that there are $\alpha_{1},\alpha
_{2}\in A(\mathcal{B}^{\#})$ with $\alpha=\alpha_{1}\vee\alpha_{2}$ and
$\alpha=\alpha_{1}\wedge\alpha_{2}=\triangle$. Then, by repeating the same
argument as in case (a), we can show that $A$ is not a block of $\alpha_{1}$
or $\alpha_{2}$. So there should exist two subsets $A_{1},A_{2}\subseteq A$
and three elements $a_{s},a_{t},a_{v}$ such that $a_{s}\in A_{1}\setminus
A_{2}$, $a_{t}\in A_{1}\cap A_{2}$ and $a_{v}\in A_{2}\setminus A_{1}$, where
$A_{1}$ is a block of $\alpha_{1}$, $A_{2}$ is a block of $\alpha_{2}$, and
$s,t,v\in\{1,...,k+1\}$, $s<t<v$. As $a_{s},a_{t}\in A_{1}$, by examining the
cycle $\mathcal{C}_{1}$ as in case (a), we get that $E_{t}=\{b_{t},c_{t}\}$ is
a block of $\alpha_{1}$.

If $v\neq k+1$, i.e. $a_{v}\neq a_{k+1}=d$, then by using the cycle
$\mathcal{C}_{2}$ as in case (a), we get that $E_{t}$ is also a block of
$\alpha_{2}$, which is a contradiction to $\alpha_{1}\wedge\alpha
_{2}=\triangle$.

Suppose $v=k+1$, i.e. $a_{v}=d$. Then $a_{t},a_{v}\in A_{2}$ and $c_{t},d\in
C$. Hence in the cycle $a_{t},A_{2},d,C,c_{t},E_{t},b_{t},F_{t},a_{t}$ (see
Fig. 4.2) only the edges $A_{2}$ and $E_{t}$ are not blocks of the atom
$\beta$. As $A_{2}$ is a block of $\alpha_{2},$ by Proposition~\ref{prop. 3.3} we get that
$E_{t}$ should be also a block of $\alpha_{2}$, in contradiction to
$\alpha_{1}\wedge\alpha_{2}=\triangle$, again.\medskip

\medskip
\noindent(ii). We partition $U$ into three disjoint sets $A=\{a_{1}%
,...,a_{k}\}$, $B=\{b_{1},...,b_{k}\}$, $C=\{c_{1},...,c_{k}\}$, as in the
case (a) in (i). We define $\alpha,\beta,\gamma\in\ $Part$(U)$ such that the
blocks of $\alpha$ are $A^{\prime}:=\{a_{1},...,a_{k-1}\}$ and $E_{1}%
=\{b_{1},c_{1}\},...,E_{k-1}=\{b_{k-1},c_{k-1}\},$ $E_{k}=\{a_{k},b_{k}\}$,
the blocks of $\beta$ are $C^{\prime}=\{c_{1},...,c_{k-1}\}$ and
$F_{1}=\{a_{1},b_{1}\},...,F_{k-1}=\{a_{k-1},b_{k-1}\},$ $F_{k}=\{b_{k}%
,c_{k}\}$, and the blocks $\gamma$ are $G_{1}=\{a_{k-1},a_{k}\}$,
$G_{2}=\{c_{k-1},c_{k}\}$, respectively any singleton$\{x\}$, $x\notin
G_{1}\cup G_{2}$. Denote $D:=U\setminus\{a_{k},b_{k},c_{k}\}$ (see Fig. 4.3).

\begin{figure}[H]
\centering
\begin{adjustbox}{width=0.4\textwidth}
\begin{tikzpicture}
\coordinate (a1) at (0, 4);
\coordinate (a2) at (1, 4);
\coordinate (ak-1) at (3, 4);
\coordinate (ak) at (4, 4);
\coordinate (b1) at (0, 2);
\coordinate (b2) at (1, 2);
\coordinate (bk-1) at (3, 2);
\coordinate (bk) at (4, 2);
\coordinate (c1) at (0, 0);
\coordinate (c2) at (1, 0);
\coordinate (ck-1) at (3, 0);
\coordinate (ck) at (4, 0);
\draw[color=blue] (ak-1) -- node[pos=0.65, circle, draw, inner sep=0.5pt, below = 0.5mm] {$\gamma$} (ak);
\draw[color=blue] (ck-1) -- node[pos=0.65, circle, draw, inner sep=0.5pt, above = 0.5mm] {$\gamma$} (ck);
\draw (ak) -- node[midway, right = -0.5mm] {$\alpha$} (bk);
\draw (c1) -- node[midway, left = -0.5mm] {$\alpha$} (b1);
\draw (c2) -- node[midway, right = -0.5mm] {$\alpha$} (b2);
\draw (ck-1) -- node[midway, left = -0.5mm] {$\alpha$} (bk-1);
\draw[color=red] (a1) -- node[color=red, midway, left = -0.5mm] {$\beta$} (b1);
\draw[color=red] (a2) -- node[color=red, midway, right = -0.5mm] {$\beta$} (b2);
\draw[color=red] (ak-1) -- node[color=red, midway, left = -0.5mm] {$\beta$} (bk-1);
\draw[color=red] (ck) -- node[color=red, midway, right = -0.5mm] {$\beta$} (bk);
\draw node[above = 1mm of a1] {$a_1$};
\draw node[above = 1mm of a2] {$a_2$};
\draw node[above left = 0.3mm and -3mm of ak-1] {$a_{k-1}$};
\draw node[above = 1mm of ak] {$a_k$};
\draw node[below = 1mm of c1] {$c_1$};
\draw node[below = 1mm of c2] {$c_2$};
\draw node[below left = 0.3mm and -3mm of ck-1] {$c_{k-1}$};
\draw node[below = 1mm of ck] {$c_k$};
\draw node[left = 1mm of b1] {$b_1$};
\draw node[left = 1mm of b2] {$b_2$};
\draw node[below left = 0.3mm and -1mm of bk-1] {$b_{k-1}$};
\draw node[right = 1mm of bk] {$b_k$};
\draw[draw=none] (a2) -- node{$\cdots$} (ak-1);
\draw[draw=none] (b2) -- node[pos=0.4]{$\cdots$} (bk-1);
\draw[draw=none] (c2) -- node{$\cdots$} (ck-1);
\draw[fill=black] (a1) circle [radius=2pt];
\draw[fill=black] (a2) circle [radius=2pt];
\draw[fill=black] (ak-1) circle [radius=2pt];
\draw[fill=black] (ak) circle [radius=2pt];
\draw[fill=black] (b1) circle [radius=2pt];
\draw[fill=black] (b2) circle [radius=2pt];
\draw[fill=black] (bk-1) circle [radius=2pt];
\draw[fill=black] (bk) circle [radius=2pt];
\draw[fill=black] (c1) circle [radius=2pt];
\draw[fill=black] (c2) circle [radius=2pt];
\draw[fill=black] (ck-1) circle [radius=2pt];
\draw[fill=black] (ck) circle [radius=2pt];
\draw[line width = 0.3mm, rounded corners, blue] ([xshift=-5mm,yshift=8mm]a1) rectangle ([xshift=5mm,yshift=-6mm]ak);
\draw[rounded corners, blue] ([xshift=-2.5mm,yshift=6.5mm]a1) rectangle ([xshift=2.5mm,yshift=-4.5mm]ak-1);
\draw[rounded corners, green] ([xshift=-7mm,yshift=5mm]b1) rectangle ([xshift=7mm,yshift=-7mm]bk);
\draw[line width = 0.3mm, rounded corners, red] ([xshift=-5mm,yshift=7mm]c1) rectangle ([xshift=5mm,yshift=-7mm]ck);
\draw[rounded corners, red] ([xshift=-2.5mm,yshift=5.5mm]c1) rectangle ([xshift=2.5mm,yshift=-5.5mm]ck-1);
\draw[line width = 0.5mm, dashed, rounded corners, gray] ([xshift=-8mm,yshift=10mm]a1) rectangle ([xshift=3mm,yshift=-10mm]ck-1);
\draw[blue] node[below right = -1mm and 1mm of a1] {$A'$};
\draw[blue] node[right = 6mm of ak] {$A$};
\draw[green] node[right = 8mm of bk] {$B$};
\draw[red] node[right = 6mm of ck] {$C$};
\draw[red] node[above right = 0.2mm and 2mm of c1] {$C'$};
\draw[gray] node[above left = 8mm and 8mm of b1] {$D$};
\end{tikzpicture}
\end{adjustbox}
\caption{Partitions in statement (ii)}%
\label{fig:casenew}%
\end{figure}
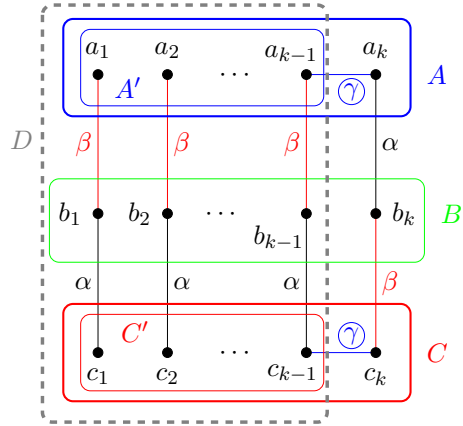

Then $\alpha\wedge\beta=\alpha\wedge\gamma=\beta\wedge\gamma=\triangle$, the
blocks of $\alpha\vee\beta$ are $D$ and $\{a_{k},b_{k},c_{k}\}$, the blocks of
$\alpha\vee\gamma$ are $A\cup\{b_{k}\}$, $E_{1},...,E_{k-1},$ and
$\{b_{k-1},c_{k-1},c_{k}\}$, respectively the blocks of $\beta\vee\gamma$ are
$C\cup\{b_{k}\}$, $F_{1},...,F_{k-1}$, and $\{b_{k-1},a_{k-1},a_{k}\}$. Then
$\alpha\vee\beta\vee\gamma=\triangledown$, and it is easy to check that
$(\alpha\vee\beta)\wedge(\alpha\vee\gamma)=\alpha$, and symmetrically,
$(\alpha\vee\beta)\wedge(\beta\vee\gamma)=\beta$. As the only blocks of
$(\alpha\vee\gamma)\wedge(\beta\vee\gamma)$ are $\{a_{k-1},a_{k}\}$ and
$\{c_{k-1},c_{k}\}$, we obtain $(\alpha\vee\gamma)\wedge(\beta\vee
\gamma)=\gamma$. These relations also imply $(\alpha\vee\beta)\wedge
\gamma=(\beta\vee\gamma)\wedge\alpha=(\alpha\vee\gamma)\wedge\beta=\triangle$.
(For instance, we get $(\alpha\vee\beta)\wedge\gamma=(\alpha\vee\beta
)\wedge(\alpha\vee\gamma)\wedge\gamma=\alpha\wedge\gamma=\triangle$.) Thus the
joins of $\alpha,\beta$ and $\gamma$, together with $\triangle$ form a Boolean
sublattice $\mathcal{B}$ of Part$(U)$ (isomorphic to a Boolean cub). 

\smallskip
Next, consider the hypergraph $\mathcal{H}$ assigned to the blocks of $\alpha,\beta$
and $\gamma$, and assume by contradiction that $\mathcal{B}$ has a proper
Boolean extension in Part$(U)$. Then by Corollary~\ref{cor 2.7}, one of its atoms
$\alpha,\beta$ and $\gamma$ is a join of two partitions $\delta_{1},\delta
_{2}\in\ $Part$(U)\setminus\{\triangle\}$ with $\delta_{1}\wedge\delta
_{2}=\triangle$. Observe that $\gamma=\delta_{1}\vee\delta_{2}$ is not
possible, since in that case either $G_{1}=\{a_{k-1},a_{k}\}$ or
$G_{2}=\{c_{k-1},c_{k}\}$ would be the single block of $\delta_{1}$ contained
in the cycle $\mathcal{C}$ through the points $a_{k-1},a_{k},b_{k}%
,c_{k},c_{k-1},b_{k-1},a_{k-1}$ in $\mathcal{H}$, which is excluded. Due
to the symmetry of $\mathcal{H}$, w.l.o.g. we may assume $\alpha=\delta_{1}%
\vee\delta_{2}$. However, according to the proof in case (a) of (i), now the
blocks $A^{\prime}:=\{a_{1},...,a_{k-1}\}$ and $E_{1}=\{b_{1},c_{1}%
\},...,E_{k-1}=\{b_{k-1},c_{k-1}\}$ must belong to the same partition $\pi
\in\{\delta_{1},\delta_{2}\}$, say $\delta_{1}$. As $\delta_{1}\neq\alpha$,
and because the two-element block $E_{k}=\{a_{k},b_{k}\}$ cannot be a proper
union of some blocks of $\delta_{1}$ and $\delta_{2}$, it must be the only
block of $\delta_{2}$. So now $E_{k}$ is the single block of $\delta_{2}$
contained in the cycle $\mathcal{C}$, which is a contradiction again. Thus
$\mathcal{B}$ is a maximal Boolean sublattice of Part$(U)$ with $8$
elements.
\end{proof}

\medskip

\noindent Let $A\subseteq U$ be a nonempty set. Then for any partition $\pi
\in\ $Part$(A)$ we denote by $\pi\cup\triangle_{U}$ its extension to Part$(U)$
obtained by adding all the sets $\{u\}$, $u\in U\setminus A$ to $\pi$ as
singleton blocks.

\medskip

\begin{theorem}\label{Thm. 4.11} Let $\mid U\mid=n$. 

\noindent (i) If $n\geq4$, then for
every integer $3\leq d\leq n-1$ there exists a maximal Boolean sublattice of
Part$(U)$ of size $2^{d}$.

\noindent(ii) If $n\geq3$ and $n$ is congruent to $0$ or $1$ mod $3$, then for every
integer $2\leq d\leq n-1$ there exists a maximal Boolean sublattice of
Part$(U)$ of size $2^{d}$.
\end{theorem}

\begin{proof} Note that proving (i) is sufficient, because then
(ii) also follows by Lemma~\ref{lemma 4.10}(i) and Example~\ref{ex. 4.1}. In view of Proposition~\ref{prop. 1.4}, (i) is clear for $d=n-1$. Suppose that $d\leq n-2$, and choose a set
$A\subseteq U$ with $p$ elements, where $3\leq p\leq n$. 

First, we prove that
any maximal Boolean sublattice $\mathcal{B}_{0}$ of Part$(A)$ with dimension
$2\leq t\leq p-1$ can be extended to a maximal Boolean sublattice
$\mathcal{B}$ of Part$(U)$ with dimension $t+n-p$. As this is obvious for
$p=n$, we may suppose $p<n$. Then $U\setminus A$ has $n-p\geq1$ elements. Let
$\alpha_{1},...,\alpha_{t}$ (with $t\geq2$) be the atoms of $\mathcal{B}_{0}$.
We extend them to the partitions $\overline{\alpha}_{1}=\alpha_{1}%
\cup\triangle_{U},...,$ $\overline{\alpha}_{t}=\alpha_{t}\cup\triangle_{U}$ of
$U$. Now, by setting any $a\in A$, we define a path $\mathcal{P\colon}$
$v_{1},e_{1},v_{2},e_{2},...,v_{n-p},e_{n-p},v_{n-p+1}$, with $v_{n-p+1}:=a$,
covering all elements $v\in U\setminus A$ as nodes, where $e_{i}%
=\{v_{i},v_{i+1}\}$, $1\leq i\leq n-p$. Then each partition $\pi
_{\{v_{i},v_{i+1}\}}$ is an atom in Part$(U)$. Let us now consider the
disjoint system $\mathcal{D}=\{\overline{\alpha}_{1},...,\overline{\alpha}%
_{t},\pi_{\{v_{1},v_{2}\}},...,\pi_{\{v_{n-p},a\}}\}$ formed by $t+n-p$
disjoint partitions and define the hypergraph $\mathcal{H}\left(
\mathcal{D}\right)  $ whose edges are their blocks. We prove that the circuit
condition holds in $\mathcal{H}\left(  \mathcal{D}\right)  $. In view of
Remark~\ref{rem. 3.6}, to check this, it is enough to consider two paths $\mathcal{P}%
_{1}$ and $\mathcal{P}_{2}$ with common endpoints such that they form a cycle
$\mathcal{C}$ in $\mathcal{H}\left(  \mathcal{D}\right)  $. Then, by our
construction, all their points should be contained in the set $A$ and their
edges can only be some edges of the hypergraph $\mathcal{H}\left(
A(\mathcal{B}_{0})\right)  $. 

\noindent As $\mathcal{B}_{0}$ is a Boolean lattice,
the circuit condition holds for any cycle in $\mathcal{H}\left(  A(\mathcal{B}%
_{0})\right)  $. Since $\mathcal{H}\left(  A(\mathcal{B}_{0})\right)  $ is the
restriction of the hypergraph $\mathcal{H}\left(  A(\mathcal{D})\right)  $ to
the set $A$, the circuit condition is valid also for the union $\mathcal{C}$
of $\mathcal{P}_{1}$ and $\mathcal{P}_{2}$ in $\mathcal{H}\left(
A(\mathcal{D})\right)  $. Then, by Theorem~\ref{Thm. 3.5}, $\mathcal{D}$ is an
independent set, and hence, $\mathcal{D}_{\vee}\cup\{0\}$ is a Boolean lattice
$\mathcal{B}$ whose atoms are elements of $\mathcal{D}$, according to
Lemma~\ref{lemma 2.5}. Therefore, the dimension of $\mathcal{B}$ is $t+n-p$. We claim that
$\mathcal{B}$ has no proper Boolean extension. Since those elements of
$\mathcal{D}$ which are not atoms in Part$(U)$ are $\overline{\alpha}%
_{1},...,\overline{\alpha}_{t}$, by Corollary~\ref{cor 2.7} any immediate Boolean
extension of $\mathcal{B}$ can be obtained only by splitting one of them, say
$\overline{\alpha}_{i}$ into two new disjoint partitions, i.e. such that
$\overline{\alpha}_{i}=\pi_{1}\vee\pi_{2}$ where $\pi_{1},\pi_{2}\in
\ $Part$(U)$, and $\pi_{1}\wedge\pi_{2}=\triangle_{U}$. As $\pi_{1},\pi
_{2}\leq\overline{\alpha}_{i}$, and all the blocks of $\overline{\alpha}_{i}$
not included in $A$ are singletons, the same holds for the blocks of $\pi_{1}$
and $\pi_{2}$. Thus they have the form $\pi_{1}=\pi_{1}^{\prime}\cup
\triangle_{U}$, $\pi_{2}=\pi_{2}^{\prime}\cup\triangle_{U}$, where $\pi
_{1}^{\prime}$ and $\pi_{2}^{\prime}$ are their restrictions to the set $A$.
Then $\pi_{1}^{\prime}\wedge\pi_{2}^{\prime}=\triangle_{A}$, and $\alpha
_{i}=\pi_{1}^{\prime}\vee\pi_{2}^{\prime}$. As $\mathcal{B}_{0}\mathcal{\ }$is
a maximal Boolean sublattice in Part$(A)$, such a decomposition of an
$\alpha_{i} \in A(\mathcal{B}_{0})$ $(1\leq i\leq t)$ could not exist. Thus there is no proper
Boolean extension of $\mathcal{B}$.

If $n-d+2$ is congruent to $0$ or $1$ mod $3$, then choose a set $A\subseteq
U$ with $p:=n-d+2$ elements. Then $4\leq p\leq n$, and $U\setminus A$ has
$n-p=d-2$ elements. Now, in view of Lemma~\ref{lemma 4.10}(i), Part$(A)$ contains a
maximal Boolean sublattice of dimension $t=2$, and this can be extended to
a maximal Boolean sublattice $\mathcal{B}$ of Part$(U)$ whose dimension is
$t+n-p=2+d-2=d$, where $d\geq t=2$.

If $n-d+2=3l+2$, where $l\geq1$, then take a subset $A$ with $p:=n-d+3=3(l+1)$
elements. Then $6\leq p\leq n$, and $U\setminus A$ has $n-p=d-3$ elements.
Now, in view of Lemma~\ref{lemma 4.10}(ii), Part$(A)$ contains a maximal Boolean
sublattice of dimension $t=3$, and this can be extended to a maximal Boolean
sublattice $\mathcal{B}$ of Part$(U)$ of dimension $t+n-p=3+d-3=d$, where
$d\geq t=3$.

Thus, in both of the above cases, Part$(U)$ contains a maximal Boolean sublattice
of size $2^{d}$, where $3\leq d\leq n-1$. So (i) is true, and this completes
the proof. 
\end{proof}

\medskip

\begin{remark}\label{rem. 4.12} We do not know wether statement (ii) of the above
theorem can be strengthened to include the case when $d=2$ and $n$ is
congruent to $2$ mod $3$, with the exception of some $n$. The particular (and
possibly even unique) exception $n=5$ is unavoidable because in the partition
lattice on a $5$-element set every $4$-element Boolean sublattice is contained
in an $8$-element Boolean sublattice. Characterizing maximal $4$-element
Boolean sublattices in general would amount a better understanding of the
various types of complements in a partition lattice.
\end{remark}

\medskip
\printbibliography

\end{document}